\documentclass[a4paper, oneside]{amsart}
\usepackage[T1]{fontenc}
\usepackage{amsbsy,amssymb,amsmath,amsthm,mathrsfs,mathtools, amsfonts}
\usepackage[shortlabels]{enumitem}
\usepackage[colorlinks, citecolor = blue]{hyperref}

\newcommand{\dmu}{\,\mathrm{d}\mu}

\newcommand{\dx}{\,\mathrm{d}x}
\newcommand{\dt}{\,\mathrm{d}t}

\newcommand{\dy}{\,\mathrm{d}y}
\newcommand{\dnu}{\,\mathrm{d}\nu}

\newcommand{\R}{\ensuremath{\mathbb{R}}}

\newcommand{\qt}[0]{\tilde{Q}}
\newcommand{\st}[0]{\tilde{S}}
\newcommand{\br}{\mathbb{R}}
\newcommand{\bc}{\mathbb{C}}
\newcommand{\bn}{\mathbb{N}}
\newcommand{\be}{\mathbb{E}}
\newcommand{\bz}{\mathbb{Z}}

\newcommand{\cd}{\mathcal{D}}

\newcommand{\cf}{\mathcal{F}}

\newcommand{\cs}{\mathcal{S}}

\newcommand{\mc}{\mathcal}

\DeclarePairedDelimiter\abs{\lvert}{\rvert}

\DeclarePairedDelimiter\cbrace\{\}

\DeclarePairedDelimiter{\ip}\langle\rangle
\newcommand{\angles}[1]{\ip{#1}}
\DeclarePairedDelimiter{\nrm}\lVert\rVert
\newcommand{\norm}[1]{\nrm{#1}}

\newcommand{\nrmb}[1]{\bigl\|#1\bigr\|}
\newcommand{\absb}[1]{\bigl|#1\bigr|}

\newcommand{\hab}[1]{\bigl(#1\bigr)}
\newcommand{\cbraceb}[1]{\bigl\{#1\bigr\}}
\newcommand{\ipb}[1]{\bigl\langle#1\bigr\rangle}

\newcommand{\nrms}[1]{\Bigl\|#1\Bigr\|}
\newcommand{\abss}[1]{\Bigl|#1\Bigr|}
\newcommand{\has}[1]{\Bigl(#1\Bigr)}
\newcommand{\cbraces}[1]{\Bigl\{#1\Bigr\}}

\DeclareMathOperator{\dist}{dist}

\DeclareMathOperator{\BMO}{BMO}
\DeclareMathOperator{\ind}{\mathbf{1}}
\DeclareMathOperator{\supp}{supp}

\DeclareMathOperator{\loc}{loc}

\newcommand{\dd}{\,\mathrm{d}}

\theoremstyle{plain}
\newtheorem{theorem}{Theorem}[section]

\newtheorem{lemma}[theorem]{Lemma}
\newtheorem{proposition}[theorem]{Proposition}
\newtheorem{corollary}[theorem]{Corollary}
\newtheorem{TheoremLetter}{Theorem}
{}

\theoremstyle{definition}
\newtheorem{definition}[theorem]{Definition}

\theoremstyle{remark}
\newtheorem{remark}[theorem]{Remark}

\newtheorem*{example}{Example}
\numberwithin{equation}{section}

\allowdisplaybreaks

\begin{document}
\title[Weighted $L^p\to L^q$-boundedness of  commutators and paraproducts]{Weighted $L^p\to L^q$-boundedness of commutators and paraproducts in the Bloom setting} 
\begin{abstract}

As our main result, we supply the missing characterization of the $L^p(\mu)\to L^q(\lambda)$ boundedness of the commutator of a non-degenerate Calder\'on--Zygmund operator $T$ and pointwise multiplication by $b$ for exponents $1<q<p<\infty$ and Muckenhoupt weights $\mu\in A_p$ and $\lambda\in A_q$.
Namely, the commutator $[b,T]\colon L^p(\mu)\to L^q(\lambda)$ is bounded if and only if $b$ satisfies the following new, cancellative condition: $$M^\#_\nu b\in L^{pq/(p-q)}(\nu),$$ where
$M^\#_\nu b$ is the weighted  sharp maximal function defined by $$
M^\#_\nu b:=\sup_{Q} \frac{\ind_Q}{\nu(Q)} \int_{Q} |b-\langle b\rangle_Q |\,\mathrm{d}x$$
and $\nu$ is the Bloom weight defined by $\nu^{1/p+1/q'}:= \mu^{1/p} \lambda^{-1/q}$. 

In the unweighted case $\mu=\lambda=1$, by a result of Hyt\"onen the boundedness of the commutator $[b,T]$ is, after factoring out constants, characterized by the boundedness of  pointwise multiplication by $b$, which  amounts to the non-cancellative condition
$b\in L^{pq/(p-q)}$. 
We provide a counterexample showing that this characterization breaks down in the weighted case $\mu\in A_p$ and $\lambda\in A_q$. Therefore, the introduction of our new, cancellative condition is necessary.

In parallel to commutators, we also characterize the weighted boundedness of dyadic paraproducts $\Pi_b$ in the missing exponent range $p\neq q$. Combined with 
previous results in the complementary exponent ranges, our results complete the characterization of the weighted boundedness of both commutators and of paraproducts for all exponents $p,q \in (1,\infty)$.
\end{abstract}

\author{Timo S. H\"anninen}
\address{Timo S. H\"anninen, Department of Mathematics and Statistics, University of Helsinki, P.O. Box 68, FI-00014 Helsinki, Finland}
\email{timo.s.hanninen@helsinki.fi}
\author{Emiel Lorist}
\address{Emiel Lorist, Delft Institute of Applied Mathematics,
Delft University of Technology, P.O. Box 5031, 2600GA Delft, The
Netherlands} 
\email{e.lorist@tudelft.nl}
\author{Jaakko Sinko}
\address{Jaakko Sinko, Department of Mathematics and Statistics, University of Helsinki, P.O. Box 68, FI-00014 Helsinki, Finland}
\email{jaakko.sinko@hotmail.com}

\maketitle
 
\setcounter{tocdepth}{1}
\tableofcontents

\section{Introduction}
\subsection{Commutators}
The boundedness of the commutator $[b,T]$ of a Calder\'on--Zygmund singular integral operator $T$  and pointwise multiplication by $b \in L^1_{\loc}$ plays an important role in harmonic analysis, which dates back to the work of Nehari \cite{Nehari1957} on the boundedness of $[b,H]$ on $L^p$, where $H$ denotes the Hilbert transform. This result was later extended  by Coifman, Rochberg, and Weiss \cite{coifman1976}, who showed that $[b,T]\colon L^p \to L^p$ is bounded for a wide class of Calder\'on--Zygmund operators $T$ with convolution kernels, $b \in \BMO$ and $p \in (1,\infty)$. Subsequent works have yielded a full characterization of those $b \in L^1_{\loc}$ for which $[b,T]\colon L^p \to L^q$ is bounded
for non-degenerate Calder\'on--Zygmund operators. The state-of-the-art result can be formulated as follows: For $p,q \in (1,\infty)$, $T$ a non-degenerate  Calder\'on--Zygmund operator and $b \in L^1_{\loc}$, there holds
\begin{align*}
 \nrmb{[b,T]}_{L^p \to L^q} \eqsim \begin{cases}
\nrm{b}_{\BMO} & p=q, \qquad \text{\cite{coifman1976},}\\
\nrm{b}_{\dot{C}^{0,\alpha}} \quad \frac\alpha{d} = \frac1p-\frac1q& p<q, \qquad \text{\cite{Janson1978},}\\
\nrm{b}_{\dot{L}^r}  \hspace{18pt} \frac1{r} = \frac1q-\frac1p\qquad & p>q, \qquad \text{\cite{hytonen2021},}
 \end{cases}
\end{align*}
where $\dot C^{0,\alpha}$ denotes the space of all $\alpha$-H\"older continuous functions and $\dot{L}^r$ denotes the Lebesgue space $L^r$ modulo constants.
We refer to \cite{hytonen2021} and \cite[Section 1.1]{HOS23} for a survey of all results leading up to this full characterization. An analogous characterization in a different context, namely on the boundedness of a Toeplitz type operator between Hardy spaces on the unit ball, was obtained by Pau and Per\"al\"a \cite{pau2020}.

\medskip

Weighted versions of commutator estimates in the case $p=q$ date back to the work of Str\"omberg, who gave an alternative proof of the boundedness of $[b,H]$ on $L^p$ using the Fefferman--Stein maximal function, which also showed the boundedness of $[b,H]$ on $L^p(w)$ for any Muckenhoupt weight $w \in A_p$ (cf. \cite[p. 419]{To86}). A sharp version of this result in terms of the weight characteristic $[w]_{A_p}$ was obtained by Chung \cite{Ch11}, which  was afterwards extended to any Calder\'on--Zygmund operator by Chung, Pereyra and Perez \cite{CPP12}. In particular, they showed that any Calder\'on--Zygmund operator $T$ satisfies
\begin{equation}\label{eq:sharponeweight}
    \nrmb{[b,T]}_{L^p(w) \to L^p(w)} \lesssim [w]_{A_p}^{2\max\cbrace{1,\frac1{p-1}}} \nrm{b}_{\BMO}
\end{equation}
for all $w \in A_p$ and $p \in (1,\infty)$ and the dependence on $[w]_{A_p}$ is sharp.

These results have been generalized to a two-weight setting by taking weights $\mu,\lambda \in A_p$ and asking for a characterization of those $b \in L^1_{\loc}$ such that $[b,T]$ is bounded from $L^p(\mu)$ to $L^p(\lambda)$. In this setting, long before even the one-weight works \cite{Ch11, CPP12}, Bloom \cite{Bloom1985} proved that for $\mu,\lambda \in A_p$ and $p \in (1,\infty)$, it holds that $[b,H] \colon L^p(\mu) \to L^p(\lambda)$ is bounded if and only if $b \in \BMO_\nu$. Here $\nu := \mu^{1/p}\lambda^{-1/p}$ and $\BMO_\nu$ is defined as the space of all $f \in L^1_{\loc}$ such that
$$
\nrm{f}_{\BMO_\nu}:= \sup_Q \frac{1}{\nu(Q)} \int_Q \abs{f - \ip{f}_Q  } \dd x<\infty,
$$
where $\ip{f}_Q = \frac{1}{\abs{Q}} \int_Q f$ and the supremum is taken over all cubes $Q$ in $\R^d$.
Note that $\BMO_\nu = \BMO$ when $\mu=\lambda$. 
The sufficiency of $b \in \BMO_\nu$ for the boundedness of $[b,T]$ from $L^p(\mu)$ to $L^p(\lambda)$ was later extended to general Calder\'on--Zygmund operators by Segovia and Torrea \cite{ST93b}.   

\medskip

 The study of characterizing the boundedness $[b,T]\colon L^p(\mu)\to L^q(\lambda)$ of commutators (or other operators) for $p,q\in(1,\infty)$ by means of conditions on the Bloom weight $\nu$, which is connected to the weights $\mu$ and $\lambda$ by the relation $\nu^{\frac1p+\frac1{q'}} := \mu^{\frac1p} \lambda^{-\frac1q}$, is now called the \emph{Bloom setting}. In this setting, the weights $\mu$ and $\lambda$  are typically satisfying the a priori assumption  $\mu \in A_p$ and $\lambda\in A_q$. We refer to \cite[Section 6]{LOR21} for a discussion of the necessity of the conditions $\mu ,\lambda\in A_p$ for the boundedness of the commutator $[b,T]$ from $L^p(\mu)$ to $L^p(\lambda)$.

The Bloom setting has in recent years attracted renewed attention, starting with the works \cite{HLW16,HLW17} of Holmes, Lacey and Wick.  This has led to the following current state-of-the-art result: For $1<p\leq q<\infty$, $\mu \in A_p$, $\lambda \in A_q$, $T$ a non-degenerate Calder\'on--Zygmund operator and $b \in L^1_{\loc}$, there holds
\begin{align*}
 \nrmb{[b,T]}_{L^p(\mu) \to L^q(\lambda)} &\lesssim \begin{cases}
\hab{[\mu]_{A_p}[\lambda]_{A_q}}^{\max\cbrace{1,\frac{1}{p-1}}}\nrm{b}_{\BMO_\nu} & p=q, \qquad \text{\cite{LOR17},}\\
C_{\mu,\lambda} \, \nrm{b}_{\BMO_\nu^\alpha} \qquad \frac\alpha{d} = \frac1p-\frac1q& p<q, \qquad \text{\cite{HOS23},}
 \end{cases}
\intertext{and}
 \nrmb{[b,T]}_{L^p(\mu) \to L^q(\lambda)} &\gtrsim C_{\mu,\lambda} \begin{cases}
\nrm{b}_{\BMO_\nu} & p=q, \qquad \text{\cite{hytonen2021},}\\
\nrm{b}_{\BMO_\nu^\alpha} \qquad \frac\alpha{d} = \frac1p-\frac1q& p<q, \qquad \text{\cite{HOS23}.}
 \end{cases}
\end{align*}
Here $\nu^{\frac1p+\frac1{q'}} := \mu^{\frac1p} \lambda^{-\frac1q}$ and $\BMO_\nu^\alpha$ is defined as the space of all $f \in L^1_{\loc}$ such that
$$
\nrm{f}_{\BMO_\nu^\alpha}:= \sup_Q \frac{1}{\nu(Q)^{1+\frac{\alpha}{d}}} \int_Q \abs{f - \ip{f}_Q  } \dd x <\infty,
$$
where the supremum is taken over all cubes $Q$ in $\R^d$.
We note that these results recover the unweighted setting for $p\leq q$, since  $\BMO_1^\alpha = \dot{C}^{0,\alpha}$. Moreover, we note that $\BMO_{\nu}^0 = \BMO_\nu$.
We refer to \cite[Section 1.2]{HOS23} for a survey of the results leading up to this characterization.

\medskip

Any characterization for the off-diagonal (upper triangular) case $p>q$  in the Bloom setting has been missing so far. In this paper, we will characterize the boundedness of $[b,T]$ for non-degenerate Calder\'on--Zygmund operators in this case $1<q<p<\infty$ by introducing a new cancellative condition on $b$. This completes the characterization of the boundedness of $[b,T]$ from $L^p(\mu)$ to $L^q(\lambda)$ for all $p,q\in(1,\infty)$ in the Bloom setting. The case $q<p$ is important, for example, for questions related to the Jacobian problem (see, e.g., \cite{Hy21c, Li17c, Li22}).

Fix $1<q<p<\infty$ and $\frac1r: = \frac1q-\frac1p$. In the unweighted case, the sufficiency of the condition  $b \in L^r$ for the boundedness of $[b,T]$ is very simple. Indeed, this is a direct consequence of the boundedness of $T$ on $L^p$ and $L^q$ and H\"older's inequality:
\begin{align*}
    \nrm{[b,T]f}_{L^q} \leq \nrm{b Tf}_{L^q} + \nrm{T(bf)}_{L^q} &\leq \nrm{b}_{L^r} \nrm{Tf}_{L^p} + \nrm{T}_{L^q \to L^q} \nrm{bf}_{L^q}\\ &\leq \hab{\nrm{T}_{L^p \to L^p} +\nrm{T}_{L^q \to L^q} } \nrm{b}_{L^r} \nrm{f}_{L^p}.
\end{align*}
The surprising fact from \cite{hytonen2021} is that $b \in \dot{L}^r$ is actually also necessary for the boundedness of $[b,T]$ from $L^p$ to $L^q$, proving that, after factoring out constants, there is no mutual cancellation between the two terms of the commutator. This is in stark contrast with the case $p=q$ and thus $r=\infty$, in which case the two individual terms are only bounded if $b \in L^\infty$ but combined one can allow $b \in \BMO$ by the seminal result of Coifman, Rochberg and Weiss \cite{coifman1976}.

In the weighted setting, it is well known that a Calder\'on--Zygmund operator $T$ is bounded on $L^p(\mu)$ for $\mu\in A_p$. Translating the observations above to the Bloom setting, one may therefore conjecture that the boundedness of $[b,T]$ from $L^p(\mu)$ to $L^q(\lambda)$ for $\mu \in A_p$ and $\lambda \in A_q$ can be characterized by the boundedness of the multiplication map $f \mapsto bf$, modulo constants. By H\"older's inequality, one easily checks that
$$\nrm{f \mapsto bf}_{L^p(\mu)\to L^q(\lambda)} = \nrm{b\nu^{-1}}_{L^r(\nu)},$$
which would suggest the right-hand side, modulo constants, as the canonical condition on $b$ for the boundedness of $[b,T]$ from $L^p(\mu)$ to $L^q(\lambda)$. However, although this condition is clearly sufficient for the boundedness of $[b,T]$, it turns out to be non-necessary, as we shall prove in Section \ref{sec:nonneccessary}. Thus, in the Bloom setting,  $[b,T]$ can be bounded from $L^p(\mu)$ to $L^q(\lambda)$ for more $b \in L^1_{\loc}$ than those $b$ for which the individual terms of the commutator are bounded, i.e. cancellation plays a role in the Bloom setting. To characterize boundedness of $[b,T]$ from $L^p(\mu)$ to $L^q(\lambda)$  we therefore have to introduce a new, cancellative condition on $b$.

To state our condition on $b$, let us introduce a weighted sharp maximal function. For $b \in L^1_{\loc}$ and a weight $\nu$, we define
$$
M^\#_\nu b:=\sup_{Q} \frac{\ind_Q}{\nu(Q)} \int_Q |b-\angles{b}_Q |\dx.$$
We will show that boundedness of $[b,T]$ from $L^p(\mu)$ to $L^q(\lambda)$ for $\mu \in A_p$ and $\lambda \in A_q$ for $p>q$ can be characterized by the assumption $M^\#_\nu b \in L^r(\nu)$ with $\nu^{1/p+1/q'} := \mu^{1/p}\lambda^{-1/q}$. Note that when $\nu=1$, and $c$ denotes a constant, we have 
$$
\nrmb{M^\#_\nu b}_{L^r(\nu)} \eqsim \inf_c \, \nrmb{(b-c)\nu^{-1}}_{L^r(\nu)},
$$
by the classical result of Fefferman and Stein \cite{FS72}, which actually extends to the case $\nu \in A_{r'}$ (see Proposition \ref{prop:feffermansteintyperesult}). However, the assumptions on $\mu$ and $\lambda$ only yield $\nu \in A_{2r'}$ (see Lemma \ref{lemma:muckenhoupt_class_bloom}), which explains why we cannot characterize the boundedness of $[b,T]$ from $L^p(\mu)$ to $L^q(\lambda)$ by
$$
\inf_c\nrmb{(b-c)\nu^{-1}}_{L^r(\nu)}<\infty.
$$

Our main result for commutators reads as follows. We refer to Subsection \ref{subsection:CZO} for the definition of a non-degenerate $\omega$-Calder\'on-Zygmund operator.

\begin{TheoremLetter}\label{thm:maincommutator}
Let $1<q<p<\infty$, $\mu \in A_p$ and $\lambda \in A_q$. Set $\frac1r := \frac1q-\frac1p$ and $\nu^{\frac1p+\frac1{q'}} := \mu^{\frac1p} \lambda^{-\frac1q}$.
Let $T$ be a non-degenerate $\omega$-Calder\'on-Zygmund operator with $\omega$ satisfying the Dini condition. For $b \in L^1_{\loc}$,
 we have
\begin{align*}
      \nrmb{[b,T]}_{L^p(\mu) \to L^q(\lambda)}&\lesssim [\mu]_{A_p}^{\max\cbrace{1,\frac{1}{p-1}}}[\lambda]_{A_q}^{\max\cbrace{1,\frac{1}{q-1}}}\nrm{M^{\sharp}_{\nu}b}_{L^r(\nu)},
      \intertext{and }
      \nrmb{[b,T]}_{L^p(\mu) \to L^q(\lambda)}&\gtrsim C_{\mu,\lambda}\nrm{M^{\sharp}_{\nu}b}_{L^r(\nu)}.
\end{align*}
\end{TheoremLetter}

The proof of the upper estimate in Theorem \ref{thm:maincommutator} can be found in Theorem \ref{theorem:uppercommutator} and the lower estimate in Corollary \ref{corollary:commutator}. These theorems actually prove more general statements:
\begin{itemize}
\item In  Theorem \ref{theorem:uppercommutator}  we also revisit the upper bound in the case $1<p<q<\infty$ from \cite{HOS23}. Indeed, we obtain a quantitative bound on $\nrm{[b,T]}_{L^p(\mu) \to L^q(\lambda)}$ in terms of $[\mu]_{A_p}$ and $[\lambda]_{A_q}$, which recovers the known sharp one-weight estimate in \eqref{eq:sharponeweight} for $\mu=\lambda$ and $q \to p$. Tracking this dependence in the proof of \cite[Theorem 2.4]{HOS23} would yield a worse dependence on $[\mu]_{A_p}$ and $[\lambda]_{A_q}$.
\item In Theorem \ref{thm:lowercommutator}  we replace the conditions $\mu \in A_p$ and $\lambda \in A_q$ in the lower bound by a weaker condition on the pair of weights $(\mu,\lambda)$. 
\end{itemize}

In a follow-up paper, we will characterize when $[b,T]$ is compact from $L^p(\mu)$ to $L^q(\lambda)$ in the case $1<q<p<\infty$, which in the unweighted setting was recently characterized by Hyt\"onen, Li, Tao and Yang in \cite{HLTY22}.  Combined with the case $1<p<q<\infty$ from \cite{HOS23}, this will also complete the characterization of the compactness of commutators in the Bloom setting.

\subsection{Paraproducts}
Paraproducts also play a vital role in harmonic analysis, for example in the celebrated dyadic representation theorem for Calder\'on--Zygmund operators by Hyt\"onen \cite{Hy12,Hy17}. 
A paraproduct $\Pi_b$ with a $b \in L^1_{\loc}$  is typically less singular than the commutator $[b,T]$ for a Calder\'on--Zygmund operator $T$. Indeed, using the aforementioned dyadic representation theorem, one can write the commutator $[b,T]$ as the sum of compositions of paraproducts and Haar shift operators (cf. \cite[Section 5]{HLW17}). Consequently, the sharp exponent in the dependence on a weight characteristic is typically smaller for paraproducts than for commutators. For example, in the one-weight setting, it is well-known (see e.g. \cite{Pe19}) that for $p \in (1,\infty)$, $w \in A_p$ and $b \in L^1_{\loc}$ the paraproduct $\Pi_b$ is bounded on $L^p(w)$ with sharp dependence on $[w]_{A_p}$ given by
\begin{equation*}
    \nrm{\Pi_b}_{L^p(w) \to L^p(w)} \lesssim [w]_{A_p}^{\max\cbrace{1,\frac{1}{p-1}}}\nrm{b}_{\BMO},
\end{equation*}
which should be compared to \eqref{eq:sharponeweight}.

In the recent work \cite{FH22} by Fragkiadaki and Holmes Fay, the Bloom setting  has also been investigated for paraproducts in the case $p=q$. They showed for $p \in (1,\infty)$,  $\mu,\lambda \in A_p$ and $\nu := {\mu}^{1/p}\lambda^{-1/p}$ that 
\begin{equation*}
    \nrm{\Pi_b}_{L^p(\mu) \to L^p(\lambda)} \lesssim [\mu]_{A_p}^{\frac{1}{p-1}}[\lambda]_{A_p}\nrm{b}_{\BMO_\nu},
\end{equation*}
which should be compared to the result for commutators in \cite{LOR17} mentioned before. Moreover, they showed that in case $p=2$ and $\lambda = \mu^{-1}$ their estimate is sharp.

\medskip

In this paper, we will give a full characterization of the boundedness of $\Pi_b$ from $L^p(\mu)$ to $L^q(\lambda)$ for $p,q \in (1,\infty)$ and $b \in L^1_{\loc}$ in the Bloom setting.
Our main result for paraproducts reads as follows. We refer to Subsection \ref{subsection:para} for the definition of the paraproduct $\Pi_b$.

\begin{TheoremLetter}\label{thm:mainparaproduct}
Let $1<p,q<\infty$, $\mu \in A_p$ and $\lambda \in A_q$. Set $\frac1r := \frac1q-\frac1p$,  $\frac\alpha{d} := \frac1p-\frac1q$ and $\nu^{\frac1p+\frac1{q'}} := \mu^{\frac1p} \lambda^{-\frac1q}$. For $b \in L^1_{\loc}$ we have
\begin{align*}
       \nrm{\Pi_b}_{L^p(\mu) \to L^q(\lambda)} &\lesssim [\mu]^{\frac{1}{p-1}}_{A_p} [\lambda]_{A_q}   \begin{cases} \nrm{b}_{\BMO_\nu^\alpha} & p\leq q,\\
     [\nu]_{A_\infty}^{1/r} \nrm{M^{\sharp}_{\nu}b}_{L^r(\nu)}  \quad &p>q,
       \end{cases}
    \intertext{and}
 \nrm{\Pi_b}_{L^p(\mu) \to L^q(\lambda)} &\gtrsim C_{\mu,\lambda} \begin{cases} \nrm{b}_{\BMO_\nu^\alpha} & p\leq q,\\
     \nrm{M^{\sharp}_{\nu}b}_{L^r(\nu)}  \quad &p>q.
     \end{cases}
\end{align*}
\end{TheoremLetter}

We note that, in the setting of Theorem \ref{thm:mainparaproduct}, we always have $\nu \in A_{2r'}$ and hence $\nu \in A_{\infty}$. The proof of the upper estimate in Theorem \ref{thm:mainparaproduct} can be found in Theorem \ref{theorem:bloompara} (see also Remark \ref{rem:sharpFH}) and the lower estimate in Theorem \ref{theorem:lowerparaproduct4}. Once again, these theorems actually prove more general statements:
\begin{itemize}
    \item Viewing a paraproduct as a bilinear operator, one would expect the upper bound in Theorem \ref{thm:mainparaproduct} to hold for the pair of weights $(\mu,\lambda)$ in a genuinely multilinear weight class (cf. \cite[Section 4.6]{ALM22}). In Theorem \ref{theorem:bloompara} we will show that it suffices to assume the weaker condition
$$
     \sup_Q \, \ipb{\mu^{1-p'}}_Q^{\frac{1}{p'}} \ip{\lambda}^{\frac1q}_Q  \ip{\nu}^{\frac{1}{p}+\frac1{q'}}_Q <\infty$$
    for the upper bound to hold. In fact,  this condition on the weights is also necessary for the stated upper bound, as we will show in Lemma \ref{lem:weightnec}.
\item For the lower estimate, we again replace the conditions $\mu \in A_p$ and $\lambda \in A_q$ in Theorem \ref{theorem:lowerparaproduct4} by a weaker condition on the pair of weights $(\mu,\lambda)$.
\end{itemize}

\subsection{Outline}
This paper is organized as follows:
\begin{itemize}
    \item In Section \ref{sec:preliminaries}, we will discuss some preliminaries on dyadic analysis, weights, Calder\'on--Zygmund operators and paraproducts.
    \item In Section \ref{sec:upper} we will prove the upper bounds for both commutators and paraproducts using sparse domination techniques.
    \item In Section \ref{sec:lower} we will prove the lower bounds for both commutators and paraproducts.
    \item Finally, in Section \ref{sec:nonneccessary} we will give an example of a $b \in L^1_{\loc}$ which does not satisfy $\inf_c\nrmb{(b-c)\nu^{-1}}_{L^r(\nu)}<\infty$, but for which the commutator $[b,T]$ and the paraproduct $\Pi_b$ are bounded.
\end{itemize}

\section{Preliminaries}\label{sec:preliminaries}
\subsection{Notation}The standard or most used notation is summarized in the following table:

\begin{tabular}{c p{0.7\textwidth}}
$Q$ & Cube with sides parallel to the coordinate axes.\\
$\ell(Q)$ & Sidelength of cube $Q$.\\
$a Q$ & Cube concentric to cube $Q$ with sidelength $a \ell(Q)$.\\
$\cd$ & Dyadic lattice.\\
$\widehat{Q}$ & Dyadic parent: Smallest cube $R \in \mc{D}$ such that $Q \subsetneq R$.\\
$\angles{f}^\mu_Q$& Average: $\angles{f}^\mu_Q:=\frac{1}{\mu(Q)}\int_Q f \dmu.$ \\
$D_Qf$ & Difference of averages: $D_Qf:=\sum_{R: \widehat{R}=Q}\angles{f}_{R}1_{R}-\angles{f}_Q1_Q$. \\

\\

$p'$ & Conjugate exponent of $p$: $\frac{1}{p'}+\frac{1}{p}:=1$.\\
$\alpha$ & Exponent defined by $\frac{\alpha}{d}:=\frac{1}{p}-\frac{1}{q}$.\\
$r$ & Exponent defined by $\frac{1}{r}:=\frac{1}{q}-\frac{1}{p}$.\\
$\mu'$ & Dual weight of $\mu$ w.r.t. exponent $p$ defined by  $\mu':=\mu^{-p'/p}$.  \\
$\lambda'$ &  Dual weight of $\lambda$ w.r.t. exponent $q$ defined by $\lambda':=\lambda^{-q'/q}$.  \\
$\nu$ & Bloom weight defined by
$
\nu^{\frac{1}{p}+\frac{1}{q'}}:=\mu^{1/p} \lambda^{-1/q}.
$\\

\\

$\nrm{f}_{\BMO_\nu^\alpha}$ & $\nrm{f}_{\BMO_\nu^\alpha}:= \sup_Q \frac{1}{\nu(Q)^{1+\frac{\alpha}{d}}} \int_Q \abs{f - \ip{f}_Q  }\dx$\\
$\norm{f}_{L^p(\mu)}$ & Lebesgue norm: $\norm{f}_{L^p(\mu)}:=\left(\int_{\br^d} \abs{f}^p \dmu \right)^{1/p}$. \\

\\

$\Pi_bf$ & Dyadic paraproduct: $\Pi_bf:=\sum_{Q\in\cd} D_Qb \angles{f}_Q$.\\
$M^\mu f$ & Maximal operator: $M^\mu f:=\sup_Q \angles{\abs{f}}_Q^\mu 1_Q$. \\
$M_\nu^\# b$ & Weighted sharp maximal operator: $$\hspace{-2cm}M_\nu^\# b:=\sup_{Q} \frac{1_Q}{\nu(Q)}\int_Q |b-\angles{b}_Q| \dx.$$  \\
\end{tabular}

If $\mu$ is the Lebesgue measure, ``$\mu$'' is omitted from the notation; similarly with $\lambda$ and $\nu$. The supremum $\sup_Q$ is either over all cubes $Q$ or over all dyadic cubes $Q\in\cd$ in a dyadic lattice $\cd$, as understood from  the context.
We will make extensive use of the notation ``$\lesssim$'' to indicate inequalities up to an implicit multiplicative constant. These constants may depend on $p,q,d$ and also on e.g. the Calder\'on--Zygmund operator $T$  and properties of its kernel and sparseness constants, but not on any of the functions under consideration. If these implicit constants depend on the weights $\mu, \lambda, \nu$, this will be denoted by ``$\lesssim_{\mu,\lambda, \nu}$''.

\subsection{Dyadic lattices}
By a {\it cube} $Q$ we mean a cube in $\R^d$ with sides parallel to the coordinate axes.

\begin{definition}[Dyadic lattice] A collection of cubes $\cd$ is called {\it a dyadic lattice} on the Euclidean space $\br^d$ if $\cd=\bigcup_{k\in \bz} \cd_k$, where $\cd_k$ is a partition of the Euclidean space  $\br^d$ by cubes of side length $2^{-k}$ and the partition $\cd_{k+1}$ refines the partition $\cd_{k}$. The cubes $Q\in \cd$ of a dyadic lattice $\cd$ are called {\it dyadic cubes}. If $\cd$ is a dyadic lattice and $Q_0\in \cd$, we denote $$\cd(Q_0):=\{Q\in \cd: Q\subseteq Q_0\}.$$
\end{definition}
From the definition, it follows that the dyadic lattice $\cd$ is a countable collection of cubes with the {\it dyadic nestedness property}: $Q\cap R \in \{Q,R,\emptyset\}$ for all $Q,R\in\cd$. For $Q \in \cd$ we define its dyadic parent $\widehat{Q}$ to be the smallest $R \in \mc{D}$ such that $Q \subsetneq R$.

\begin{example}The dyadic lattice
$$
\cd^0:=\{2^{-k}([0,1)+j): k\in \bz, j \in \bz^d\}
$$
is called {\it the standard dyadic lattice} and the dyadic lattices
$$
\cd^\alpha:=\{2^{-k}([0,1)+j+(-1)^k \alpha ): k\in \bz, j \in \bz^d\} \quad \text{where $\alpha \in \{0,\tfrac{1}{3},\tfrac23\}^d$,}
$$
are examples of {\it shifted} dyadic lattices.
\end{example}
The geometric observation that generic cubes can be approximated by shifted dyadic cubes is known as {\it the one-third trick} \cite{okikiolu1992} and can be stated as follows (e.g. \cite[Lemma 3.2.26]{HNVW16}): 
\begin{lemma}[Shifted dyadic cubes approximate generic cubes]\label{lemma:onethird}For each cube $Q\subseteq \br^d$ there are $\alpha \in \{0,\frac{1}{3},\tfrac23\}^d$ and $R\in \cd^\alpha$ such that $Q\subseteq R$ and $\ell(R)\leq 3\ell(Q)$.
\end{lemma}
Thanks to this approximation, dyadic and non-dyadic quantities are often comparable when measures are doubling. As an example that we will use in what follows, we can compare dyadic and non-dyadic sharp maximal functions.
\begin{corollary}[Dyadic and non-dyadic sharp maximal functions are comparable]\label{lemma:maximal_dyadic_nondyadic}Let $\nu$ be a doubling measure and $f\in L^1_{\mathrm{loc}}$. Then
$$
M^\#_\nu f(x)\eqsim_\nu \sup_{\alpha \in \{0,\frac{1}{3},\frac23\}^d} M^\#_{\nu,\cd^\alpha}f(x), \qquad x \in \R^d,
$$
where in $M^\#_{\nu,\cd^\alpha}f$ the defining supremum is taken over cubes $Q\in \cd^\alpha$.
\end{corollary}
\begin{proof}Let $Q$ be a cube. Then there are $\alpha \in \{0,\frac{1}{3},\frac23\}^d$ and $R\in \cd^\alpha$ such that $Q\subseteq R$ and $\ell(R)\leq 3\ell(Q)$. The observation that $R\subseteq 5Q$ together with the assumption that $\nu$ is doubling gives
$$
\frac{1}{\nu(Q)}=\frac{\nu(5Q)}{\nu(Q)} \frac{1}{\nu(5Q)}\leq C_\nu \frac{1}{\nu(R)}.
$$
Moreover, because $Q\subseteq R$, we obtain 
$$
\int_Q |b-\angles{b}_Q| \dx\leq 2 \inf_c \int_Q |b-c| \dx\leq  2 \inf_c \int_R |b-c| \dx \leq 2 \int_R |b-\angles{b}_R| \dx,
$$
which finishes the proof.
\end{proof}

\subsection{Dyadic analysis} In what follows, we introduce Carleson and sparse collections, the basic estimates for them, and the related theory of the $A_\infty$ measures. The results are well-known but perhaps somewhat scattered, so we hope that the systematic exposition in this section will be convenient for the reader.

Throughout this section an arbitrary dyadic lattice $\cd$ is fixed; for brevity ``$\cd$'' is suppressed in most notation.
A basic tool in dyadic analysis is the well-known Hardy--Littlewood maximal inequality. The {\it dyadic Hardy--Littlewood maximal operator} $M^\mu$ is defined by
$$
M^\mu f:=\sup_{Q\in\cd} \angles{\abs{f}}^\mu_Q 1_Q. 
$$
\begin{lemma}[Dyadic Hardy--Littlewood maximal inequality]Let $\mu$ be a locally finite Borel measure and $p\in(1,\infty)$. Then
$$
\norm{M^\mu}_{L^p(\mu)}\leq p' \norm{f}_{L^p(\mu)}
$$
for all $f\in L^p(\mu)$.
\end{lemma}
The Hardy--Littlewood maximal inequality and the Carleson embedding theorem are closely tied: each can be derived from the other. A family $\{s_Q\}_{Q\in \cd}$ of complex numbers indexed by dyadic cubes satisfies {\it the $\mu$-Carleson packing condition} if there is a constant $C>0$ such that $$
\sum_{Q'\in \cd: Q'\subseteq Q} |s_{Q'}|\mu(Q')\leq C \mu(Q)
$$
for all cubes $Q\in\cd$. The least constant in the estimate is called {\it the  Carleson norm} and is denoted by $\norm{s}_{\mathrm{Car} (\mu)}$, i.e.
$$\norm{s}_{\mathrm{Car} (\mu)}:=\sup_{Q\in\cd} \frac{1}{\mu(Q)}\sum_{Q'\in \cd: Q'\subseteq Q} |s_{Q'}|\mu(Q').$$

\begin{lemma}[Dyadic Carleson embedding theorem]\label{lemma:dyadic_carleson_embedding}Let $\mu$ be a locally finite Borel measure and let $p\in(1,\infty)$. Then 
$$
\has{ \sum_{Q\in \cd} |\angles{f}^\mu_Q|^p s_Q \mu(Q) }^{1/p}\leq p'  \norm{s}_{\mathrm{Car}(\mu)}^{1/p} \norm{f}_{L^p(\mu)}
$$
for all $\{s_Q\}_{Q\in\cd}\in \mathrm{Car}(\mu)$ and $f\in L^p(\mu)$.
\end{lemma}

In this context the $A_\infty$-class enters because it is precisely the class that preserves the Carleson packing condition.
A locally finite Borel measure $w$ is said to satisfy the {\it  dyadic Fujii--Wilson $A_\infty$-condition} with respect to a locally finite Borel measure $\mu$, and this is denoted by $w \in A_\infty(\mu)$, if
\begin{equation}\label{eq:temp34}
\int_Q \sup_{R\in \cd(Q):x \in R} \frac{w(R)}{\mu(R)}\dmu(x)\leq C w(Q)
\end{equation}
for all cubes $Q\in \cd$. The least admissible constant $C$ is denoted by $[w]_{A_\infty (\mu)}$ and is called {\it the dyadic Fujii-Wilson $A_\infty$-characteristic}. 

The following lemma can be found, for example, in \cite[Proposition 3.7]{hanninen2021}.

\begin{lemma}[Carleson condition is preserved by $A_\infty$ measures]\label{lemma:carleson_operator_norm}Let $w,\mu$ be locally finite Borel measures. Then
$$
[w]_{A_\infty(\mu)}=\sup_{s:=\{s_Q\}_{Q\in\cd}}\frac{\norm{s}_{\mathrm{Car} (w)}}{\norm{s}_{\mathrm{Car} (\mu)}}.
$$
\end{lemma}
This fact, combined with Lemma \ref{lemma:dyadic_carleson_embedding}, implies the weighted version of the Carleson embedding theorem:
\begin{lemma}[Weighted version of the dyadic Carleson embedding theorem]\label{lemma:dyadiccarlesonweight} Let $\mu$ be a locally finite Borel measure,  $w\in A_\infty(\mu)$ and let $p\in(1,\infty)$. Then
$$
\has{ \sum_{Q\in \cd} |\angles{f}^w_Q|^p s_Q w(Q) }^{1/p}\leq p'  [w]_{A_\infty(\mu)}^{1/p}\norm{s}_{\mathrm{Car} (\mu)}^{1/p} \norm{f}_{L^p(w)}
$$
for all $\{s_Q\}_{Q\in\cd}\in \mathrm{Car}(\mu)$ and $f\in L^p(w)$.
\end{lemma}

A collection $\cs$ of dyadic cubes is called {\it $(C,\mu)$-Carleson} if there exists a constant $C>0$ such that $$
\sum_{S'\in \cs: S'\subseteq S} \mu(S')\leq C \mu(S)
$$
for all $S\in\cs$. The associated Carleson norm is given by
$$
\norm{\cs}_{\mathrm{Car} (\mu)}:=\norm{1_\cs(Q)}_{\mathrm{Car} (\mu)}=\sup_{S\in\cs} \frac{1}{\mu(S)}\sum_{S'\in \cs: S'\subseteq S}\mu(S').
$$
Thus, a Carleson collection $\cs$ of dyadic cubes is a special case of a Carleson family of complex numbers via the correspondence $s_Q:=1_{\cs}(Q)$.  Iterated stopping time arguments typically generate sparse collections, which are collections of large disjoint parts and hence in particular satisfy the Carleson packing condition.
\begin{definition}[Sparseness]
Let $\mu$ be a locally finite Borel measure and let $\gamma\in(0,1)$. A collection $\cs$ of dyadic cubes is called $(\gamma,\mu)$-sparse if for each $S\in\cs$ there exists $E_S\subseteq S$ such that $\mu(E_S)\geq \gamma \mu(S)$ and such that the sets $\{E_S\}_{S\in\cs}$ are disjoint. If $\mu$ is the Lebesgue measure, it is omitted from the notation.
\end{definition}
 Every $\gamma$-sparse collection $\cs$ satisfies the $\gamma^{-1}$-Carleson packing condition:
$$
\norm{\cs}_{\mathrm{Car} (\mu)}:=\sup_{S\in\cs} \frac{1}{\mu(S)}\sum_{S'\in\cs: S'\subseteq S}\mu(S')\leq \frac{1}{\gamma} \sup_{S\in\cs} \frac{1}{\mu(S)}\sum_{S'\in\cs: S'\subseteq S}\mu(E_{S'})\leq \frac{1}{\gamma}.
$$
 An obstruction for the converse is a point mass: a mass point can not be divided between two sets. Excluding this obstruction, the converse also holds (see \cite[Corollary 6]{verbitsky1996}):
\begin{lemma}[Carleson and sparse are equivalent]\label{lemma:carleson_implies_sparse}Let $\mu$ be a locally finite Borel measure with no point masses. Then each $(C,\mu)$-Carleson collection of dyadic cubes is $(\tfrac1C,\mu)$-sparse. 
\end{lemma}
For different proofs and generalizations, see \cite{barron2019,cascante2017,hanninen2018,HL25, LN15, rey2022}.
As an immediate corollary of the equivalence of being Carleson and being sparse, we see that $A_\infty$-measures do not only preserve Carleson, but also sparseness:
\begin{corollary}[$A_\infty$-measures preserve sparseness] \label{cor:Ainftysparse} Let $\mu$ and $w$ be locally finite Borel measures with no point masses. Assume that $w\in A_\infty(\mu)$. Then every $(\gamma,\mu)$-sparse collection is $(\gamma[w]_{A_\infty(\mu)}^{-1},w)$-sparse.
\end{corollary}
\begin{proof}The statement follows by combining Lemma \ref{lemma:carleson_operator_norm} and Lemma \ref{lemma:carleson_implies_sparse}.
\end{proof}

 For many purposes Carleson (or in particular sparse) collections behave like disjoint collections. An instance of this is captured by the following lemma:
\begin{lemma}\label{lemma:basic_lemma_carleson}Let $\mu$ be a locally finite Borel measure. Assume that $\cs$ is a $\mu$-Carleson (or in particular $\mu$-sparse) collection of dyadic cubes. Then
$$
\nrms{\sum_{S\in\cs} a_S 1_S}_{L^p(\mu)}\leq p \norm{\cs}_{\mathrm{Car} (\mu)}^{1/p'}  \has{\sum_{S\in\cs} |a_S|^p \mu(S)}^{1/p}
$$
for all families $\{a_S\}_{S\in\cs}$ of complex numbers.
\end{lemma}
\begin{proof}By H\"older's inequality, we have
$$
\int_{\R^d} \has{\sum_{S\in\cs} a_S 1_S}g \dmu\leq \has{\sum_{S\in\cs} |a_S|^p \mu(S)}^{1/p} \has{\sum_{S\in\cs} (\angles{|g|}_{S}^\mu)^{p'} \mu(S)}^{1/p'}.
$$
From this the statement follows by $L^p-L^{p'}$ duality and the dyadic Carleson embedding theorem (Lemma \ref{lemma:dyadic_carleson_embedding}).
\end{proof}
As before, the fact that an $A_\infty$-measure preserves the Carleson packing condition leads to a weighted version of this lemma:
\begin{lemma}\label{lemma:weighted_basic}Let $\mu$ be a locally finite Borel measure and $w\in A_\infty(\mu)$. Assume that $\cs$ is a $\mu$-Carleson (or, in particular, $\mu$-sparse) collection of dyadic cubes. Then
$$
\nrms{\sum_{S\in\cs} a_S 1_S}_{L^p(w)}\leq p [w]_{A_\infty(\mu)} ^{1/p'}\norm{\cs}_{\mathrm{Car}^\cd (\mu)}^{1/p'}  \has{\sum_{S\in\cs} |a_S|^p w(S)}^{1/p}.
$$
\end{lemma}
\begin{proof}The statement follows by combining Lemma \ref{lemma:carleson_operator_norm} and Lemma \ref{lemma:basic_lemma_carleson}.
\end{proof}

When measures are doubling, the above estimates involving dyadic cubes translate easily into their counterparts involving generic cubes, by using shifted dyadic lattices (Lemma \ref{lemma:onethird}) in a typical fashion. In particular, we will need the following non-dyadic counterpart:

\begin{lemma}[Weighted basic lemma for sparse collections of generic cubes]\label{lemma:basic_lemma_sparse_generic}Let $p\in(1,\infty)$. Assume that $w$ and $\mu$ are doubling measures and that $w\in A_\infty(\mu)$. 
Assume that $\cs$ is a $(\gamma,\mu)$-sparse countable collection of cubes. Then
$$
\nrms{\sum_{S\in\cs} a_S 1_{S}}_{L^p(w)}\lesssim_{\mu,w} [w]_{A_\infty(\mu)}^{1/p'} \has{\sum_{S} |a_S|^p w(S)}^{1/p}.
$$
The implicit constant depends on the measures $\mu$ and $w$ via their doubling constants.
\end{lemma}
\begin{proof}The lemma follows from its dyadic counterpart in Lemma \ref{lemma:weighted_basic} via shifted dyadic cubes (see Lemma \ref{lemma:onethird}).
\end{proof}

\subsection{Muckenhoupt weights} 
A locally integrable function $w \colon \R^d \to (0,\infty)$ is called a \emph{weight}. By associating the measure $w(E) := \int_E w\dd x$ for measurable $E \subseteq \R^d$ to a weight $w$, the statements of the previous subsection are applicable to $w$.

 For $p\in(1,\infty)$ the (non-dyadic) {\it  $A_p$-characteristic} $[w]_{A_p}$ is defined by
$$
[w]_{A_p}:=\sup_Q \angles{w}_Q \angles{w^{-p'/p}}^{p/p'}_Q,
$$
and for $p=\infty$ the (non-dyadic) Fujii-Wilson $A_\infty$-characteristic from the previous section takes the form 
$$
[w]_{A_\infty}:=\sup_{Q} \frac{1}{w(Q)}\int_Q \sup_{R\ni x} \frac{w(Q\cap R)}{|R|}\dx,
$$
where the suprema are over all cubes. We will write $w \in A_p$ if $[w]_{A_p}<\infty$.
\begin{remark}
The dyadic version of the $A_p$-characteristic is defined similarly, but with the suprema over dyadic cubes instead of generic cubes. The dyadic version is used in the context of dyadic operators (e.g. dyadic paraproducts), while in the context of non-dyadic operators (e.g. commutators) the non-dyadic version is used; this distinction is  suppressed in the notation and understood from the context. 
\end{remark}
\begin{example}The power weight $w_\delta(x):=|x|^\delta$ satisfies $w_\delta \in A_p$ if and only if $-d<\delta< (p-1)d$.
\end{example}
We will only need the following well-known properties for the $A_p$-characteristic: 
\begin{lemma}[Duality and monotonicity of the $A_p$ characteristic]\label{lemma:muckenhoupt_basics}Let $w$ be a weight.
\begin{itemize}
\item \emph{(Duality)}  $[w]^{1/p}_{A_p}=[w^{-p'/p}]_{A_{p'}}^{1/p'}$ for $p \in (1,\infty).$
\item \emph{(Monotonicity)} 
$[w]_{A_p}\geq [w]_{A_q}$ and $[w]_{A_p}\gtrsim [w]_{A_\infty}$ for $1<p\leq q< \infty$.
\end{itemize}
\end{lemma}
\begin{proof}The duality is clear from the definitions. The inequality $
[w]_{A_p}\geq [w]_{A_q}$ for $1<p\leq q< \infty$ follows by Jensen's inequality and  for  $[w]_{A_p}\gtrsim [w]_{A_\infty}$ we refer to \cite[Proposition 2.2]{HP13}.
\end{proof}

We will use the following bound for the weighted norm of the maximal operator.
\begin{lemma}[\cite{Bu93}]\label{lemma:strongtypemaximalnondyadic}
Let $p\in(1,\infty)$ and $w\in A_p$. Then 
\[
\|Mf\|_{L^p(w)}\lesssim [w]_{A_p}^\frac{1}{p-1}\|f\|_{L^p(w)}
\]
for all $f \in L^p(w)$.
\end{lemma}

\subsection{Bloom weight and the Bloom--Muckenhoupt joint characteristics}
Let $p,q \in (1,\infty)$ and let $\mu,\lambda$ be weights. An operator $U\colon L^p(\mu)\to L^q(\lambda)$ is bounded if and only if 
$$
\abss{\int_{\R^d} (Uf)g\dx}\leq C \norm{f}_{L^p(\mu)}\norm{g}_{L^{q'}(\lambda^{-q'/q})}
$$
for all $f \in L^p(\mu)$ and $g \in L^{q'}(\lambda^{q'/q})$.
The weight $\lambda^{-q'/q}$ appearing in this bilinear estimate is called the {\it dual weight} of the weight $\lambda$ with respect to the weighted space $L^q(\lambda)$ and is denoted by $\lambda':=\lambda^{-q'/q}$. Similarly, we write $\mu':= \mu^{-p'/p}$ for the dual weight of the weight $\mu$ with respect to the weighted space $L^p(\mu)$.

\begin{definition}[Bloom weight]Let $\mu,\lambda$ be weights and let $p,q\in(1,\infty)$. The Bloom weight $\nu=\nu_{L^p(\mu)\to L^q(\lambda)}$ associated with $L^p(\mu)\to L^q(\lambda)$-boundedness is defined by 
\begin{equation}\label{temp:bloom}
\nu^{\frac{1}{p}+\frac{1}{q'}} :=\mu^{1/p} (\lambda')^{1/q'}= \mu^{1/p} \lambda^{-1/q}.
\end{equation}
\end{definition}
\begin{remark}
An operator $U\colon L^p(\mu)\to L^q(\lambda)$ is bounded if and only if $U^*\colon L^{q'}(\lambda')\to L^{p'}(\mu')$ is bounded, where $U^*$ is the adjoint with respect to the unweighted integral dual pairing. Note that the Bloom weight is invariant under duality in the sense that
$$\nu_{L^{q'}(\lambda')\to L^{p'}(\mu')}=\nu_{L^{p}(\mu)\to L^{q}(\lambda)}.$$
\end{remark}

In our main theorems, Theorems \ref{thm:maincommutator} and \ref{thm:mainparaproduct}, we assume that $\mu \in A_p$ and $\lambda \in A_q$. However, as we shall see in Sections \ref{sec:upper} and \ref{sec:lower}, one can sometimes use weaker, joint weight characteristics in the upper and lower bounds in the Bloom setting. We will introduce these joint characteristics in the following definition and afterwards show that these characteristics are finite in case $\mu \in A_p$ and $\lambda \in A_q$ and $\nu^{\frac{1}{p}+\frac{1}{q'}} =\mu^{1/p} \lambda^{-1/q}$.

\begin{definition}[Joint Bloom--Muckenhoupt characteristics]\label{def:joint}
Let $p,q \in (1,\infty)$ and let $\mu,\lambda,\nu$ be weights. 
\begin{enumerate}[(i)]
    \item (Characteristic for upper bounds) We define
$$
[\mu',\lambda,\nu]_{B_{p',q}}:=\sup_Q \, \angles{\mu'}^{1/p'}_Q \angles{\lambda}^{1/q}_Q \angles{\nu}_Q^{{1}/{p}+{1}/{q'}},$$
and write $(\mu',\lambda,\nu) \in B_{p',q}$ if $[\mu',\lambda,\nu]_{B_{p',q}}<\infty.$
    \item (Characteristic for lower bounds) We define
$$
[\mu,\lambda']_{B_{p,q'}(\nu)}:= \sup_Q \,\has{\frac{\mu(Q)}{\nu(Q)}}^{1/p} \has{\frac{\lambda'(Q)}{\nu(Q)}}^{1/q'}$$
and write $(\mu,\lambda') \in B_{p,q'}(\nu)$ if $[\mu,\lambda']_{B_{p,q'}(\nu)}<\infty.$
\end{enumerate}
\end{definition}
\begin{remark}\
\begin{enumerate}[(i)]
\item Interpreting the lower bound Bloom characteristic as an assumption on the triplet of weights $(\mu,\lambda',\nu)$ through
\begin{equation*}
  [\mu,\lambda',\nu ]_{B_{p,q'}(\nu)} :=   \sup_Q\, \has{\frac{\mu(Q)}{\nu(Q)}}^{1/p} \has{\frac{\lambda'(Q)}{\nu(Q)}}^{1/q'}\has{\frac{\nu(Q)}{\nu(Q)}}^{1/p'+1/q},
\end{equation*}
we see that the notation for the lower and upper bound Bloom characteristics could be unified, identifying the upper bound Bloom weight class with $B_{p',q}(1)$.
    \item The class $B_{p',q}$ is exactly the class that plays an important role in the weighted $L^p \times L^{q'} \to L^{r'}$-boundedness of bilinear operators for $\frac1{r'} = \frac{1}{p} +\frac{1}{q'}$ (cf. \cite{LOPTT09}).
    \item The class of weights $B_{p,q'}(\nu)$ is closely related to the Muckenhoupt classes used in weighted $L^{p'} \to L^{q'}$-boundedness of linear operators, using $\nu$ as the base measure on $\R^d$ (cf. \cite{FH18}). 
\end{enumerate}
\end{remark}

\begin{remark}\,\label{rem:Bloomsettinginduced}
Note that, for a fixed $b\in L^1_{\loc}$, the norm
$
\nrm{b}_{\BMO_\nu^\alpha}
$
decreases as $\nu$ increases. Similarly, the norm $\nrm{M_\nu^\# b}_{L^r(\nu)}$ is comparable to a discrete version (as stated in Lemma \ref{lemma:discretization}), which decreases as $\nu$ increases. Moreover, we have:
\begin{itemize}
 \item For weights $(\mu,\lambda,\nu) \in B_{p',q}$ we have by the Lebesgue differentiation theorem that 
$$
\mu'(x)^{1/p'} \lambda(x)^{1/q} \nu(x)^{1/p+1/q'} \leq [\mu',\lambda,\nu]_{B_{p',q}}, \qquad x \in \R^d 
$$
and thus
$$
\nu^{\frac{1}{p}+\frac{1}{q'}} \leq [\mu',\lambda,\nu]_{B_{p',q}} \cdot \mu^{1/p} \lambda^{-1/q}.
$$
\item Similarly, for weights $(\mu,\lambda) \in B_{p,q'}(\nu)$ we have
$$
\mu^{1/p} \lambda^{-1/q} \leq  [\mu,\lambda']_{B_{p',q}(\nu)} \cdot \nu^{\frac{1}{p}+\frac{1}{q'}}.
$$
\end{itemize}
Therefore,  we will always assume to be in the Bloom setting $\nu^{\frac{1}{p}+\frac{1}{q'}} =\mu^{1/p} \lambda^{-1/q}$  when proving upper bounds for weights in the class $B_{p',q}$ or lower bounds for weights in the class $B_{p',q}(\nu)$.
\end{remark}

As already announced, the assumption that the joint  Bloom--Muckenhoupt characteristics in Definition \ref{def:joint} are finite relaxes the typical Bloom-setting assumption that $\mu \in A_p$ and $\lambda \in A_q$ separately.
\begin{lemma}[Separate Muckenhoupt conditions imply the joint Bloom--Muckenhoupt conditions]\label{lemma:muckenhoupt_implies_bloom} Let $p,q \in (1,\infty)$, $\mu \in A_p$, $\lambda \in A_q$ and let $\nu^{\frac{1}{p}+\frac{1}{q'}} :=\mu^{1/p}\lambda^{-1/q}$. Then 
\begin{align*}
   [\mu',\lambda,\nu]_{B_{p',q}} &\leq [\mu]_{A_{p}}^{1/p} [\lambda]_{A_q}^{1/q},\\
    [\mu,\lambda']_{B_{p,q'}(\nu)}&\leq [\mu]_{A_p}^{1/p} [\lambda]_{A_{q}}^{1/q}.
\end{align*}
\end{lemma}
\begin{proof} The first estimate is a direct corollary of H\"older's inequality. The second estimate
follows from Jensen's inequality, see \cite[Proposition 3.1]{HOS23}.
\end{proof}

We end this subsection with two lemmata on properties of $\mu$ and $\lambda$ induced by the joint Muckenhoupt--Bloom conditions.

\begin{lemma}[Bloom--Muckenhoupt characteristic for upper bounds]\label{lemma:muckenhoupt_class_bloom}Let $p,q \in (1,\infty)$ and set $\frac1{r'} := \frac{1}{p} +\frac{1}{q'}$.  Let $\mu,\lambda$ be weights and set $\nu^{{1}/{r'}}:=\mu^{1/p}\lambda^{-1/q}$.
\begin{enumerate}[(i)]
\item \label{it:Bloomw1}  $(\mu',\lambda,\nu) \in B_{p',q}$ if and only if $\mu' \in A_{2p'}$, $\lambda \in A_{2q}$ and $\nu \in A_{2r'}$ and in this case
$$[\nu]_{A_{2r'}} \leq [\mu',\lambda,\nu]_{B_{p',q}}^{r'}.$$
In particular, if $\mu \in A_p$, $\lambda \in A_q$, then $\nu \in A_{2r'}$ with
$$[\nu]_{A_{2r'}}  \leq
[\mu]_{A_p}^{r'/p}[\lambda]_{A_q}^{r'/q}.$$
\item \label{it:Bloomw2} For $1<s<2r'$ there exist weights $\mu \in A_p$ and $\lambda \in A_q$ such that $\nu \notin A_s$.
\item \label{it:Bloomw3}
For each power weight $w\in A_{2r'}$ there exist power weights $\mu\in A_p$ and $\lambda\in A_q$ such that $w^{{1}/{r'}} =\mu^{1/p}\lambda^{-1/q}=:\nu^{{1}/{r'}}$.  
\end{enumerate}
\end{lemma}
\begin{proof} For \ref{it:Bloomw1} we refer to \cite[Theorem 3.6]{LOPTT09} and  Lemma \ref{lemma:muckenhoupt_implies_bloom}. Part \ref{it:Bloomw2} follows from using power weights $\mu(x):=\abs{x}^{\alpha}$ and $\lambda(x):=\abs{x}^\beta$ and choosing the exponents suitably, using the fact that a power weight $w_\delta:=\abs{x}^{\delta}$ satisfies $w_\delta\in A_s$, $1<s<\infty$, if and only if $-d<\delta<(s-1)d$. The statement in \ref{it:Bloomw3} follows similarly.
\end{proof}

\begin{lemma}[Bloom--Muckenhoupt characteristic for lower bounds]\label{lemma:ainfinity_wrt_bloom}Let $p,q\in(1,\infty)$, let $\mu,\lambda$ be weights and set
$\nu^{\frac{1}{p}+\frac{1}{q'}} :=\mu^{1/p}\lambda^{-1/q}$.
Then
\begin{align*}
    [\mu]_{A_\infty(\nu)}&\lesssim \, [\mu,\lambda']^{p}_{B_{p,q'}(\nu)},\\
    [\lambda']_{A_\infty(\nu)}&\lesssim \,[\mu,\lambda']^{q'}_{B_{p,q'}(\nu)}
\end{align*}
for the dyadic versions of the weight characteristics.
\end{lemma}
\begin{remark}Assuming that the weights $\mu,\lambda,\nu$ are doubling, these estimates hold also for the non-dyadic versions of the weight characteristics, but with the implicit constant depending on the doubling constants. This follows from the one-third trick in Lemma \ref{lemma:onethird}.
\end{remark}
Lemma \ref{lemma:ainfinity_wrt_bloom} could be deduced from viewing the class $B_{p,q'}(\nu)$ as a bilinear Muckenhoupt weight class and then using \cite[Theorem 3.6]{LOPTT09} combined with the monotonicity of Muckenhoupt classes in Lemma \ref{lemma:muckenhoupt_basics} (now considering weights with respect to the measure $\dnu$ instead of the Lebesgue measure $\dx$). For the reader's convenience, however, we write down the following, more transparent proof.

\begin{proof}[Proof of Lemma \ref{lemma:ainfinity_wrt_bloom}]
Fix a cube $Q$. By H\"older's inequality and the definition of the dyadic weight characteristic,
\begin{equation*}
\begin{split}
\int_{\R^d} &\sup_{R\subseteq Q} 1_R \frac{\mu(R)}{\nu(R)} \dnu\\
&\leq \int_{\R^d} \left(\sup_{R\subseteq Q} 1_R \frac{\mu(R)}{\nu(R)}\right)^{1/p}  \has{\sup_{R\subseteq Q} 1_R \frac{\mu(R)}{\nu(R)}}^{1/p'}\dnu\\
&\leq [\mu,\lambda']_{B_{p,q'}(\nu)} \int_{\R^d} \has{\sup_{R\subseteq Q} 1_R \frac{\lambda'(R)}{\nu(R)}}^{-1/q'}  \has{\sup_{R\subseteq Q} 1_R \frac{\mu(R)}{\nu(R)}}^{1/p'}\dnu\\
&\leq [\mu,\lambda']_{B_{p,q'}(\nu)}  \has{ \int_{\R^d} \has{\sup_{R\subseteq Q} 1_R \frac{\lambda'(R)}{\nu(R)}}^{-p/q'} \dnu}^{1/p} \has{\int_{\R^d} \sup_{R\subseteq Q} 1_R \frac{\mu(R)}{\nu(R)}\dnu }^{1/p'}.
\end{split}
\end{equation*}
Therefore, via a monotone convergence argument to ensure finiteness, we have
$$
\int_{\R^d} \sup_{R\subseteq Q} 1_R \frac{\mu(R)}{\nu(R)} \dnu\leq [\mu,\lambda']_{B_{p,q'}(\nu)}^p  \int_{\R^d} \has{\sup_{R\subseteq Q} 1_R \frac{\lambda'(R)}{\nu(R)}}^{-p/q'} \dnu.
$$
Note that
$
\lambda'(Q)=\int_Q \lambda' \nu^{-1} \dnu
$
and hence $\frac{\lambda'(Q)}{\nu(Q)}=\angles{\lambda' \nu^{-1}}_{Q}^\nu$. Let $\alpha\in (1,\infty)$ be an auxiliary exponent. Note that $t\mapsto t^{-p/(q'\alpha)}$ is convex. By Jensen's inequality, we have 
$$
\has{\frac{\lambda'(Q)}{\nu(Q)}}^{-p/q'}\leq \has{\ipb{(\lambda' \nu^{-1})^{-p/(q'\alpha)}}_{Q}^\nu}^{\alpha}.
$$
Therefore, by the dyadic Hardy--Littlewood maximal inequality, 
$$
\int_{\R^d} \has{\sup_{R\subseteq Q} 1_R \frac{\lambda'(R)}{\nu(R)}}^{-p/q'} \dnu\leq (\alpha')^{\alpha} \int_Q (\lambda' \nu^{-1})^{-p/q'} \dnu=\int_Q (\lambda' \nu^{-1})^{-p/q'}  \nu \dx.
$$
Since $(\lambda' \nu^{-1})^{-p/q'}\nu=\mu$ by assumption, choosing $\alpha=2$ yields
$$
\frac{1}{\mu(Q)}\int_{\R^d} \sup_{R\subseteq Q} 1_R \frac{\mu(R)}{\nu(R)} \dnu\leq 4\, [\mu,\lambda']_{B_{p,q'}(\nu)}^p
$$
Taking the supremum over all dyadic cubes yields the first estimate, the second is proven similarly.
\end{proof}

\subsection{Calder\'on--Zygmund operators and kernels}\label{subsection:CZO}
We now turn to the definition of the operators under consideration. We start with the definition of Calder\'on--Zygmund operators and their kernels.

A function $\omega:[0,\infty)\to[0,\infty)$  is called a {\it modulus of continuity} if it is increasing and satisfies $\lim_{t\to 0} \omega(t)=0$. It is said to satisfy {\it the Dini condition} if $\int_0^1 \frac{\omega(t)}{t} \dt <\infty$.

\begin{definition}[Non-degenerate $\omega$-Calder\'on--Zygmund kernel]\label{def:czkernel}Let $\omega:[0,\infty)\to[0,\infty)$ be a  modulus of continuity. A kernel $K:\br^d\times \br^d \setminus \{x=y\}\to \bc$ is called a {\it (two-variable) $\omega$-Calder\'on--Zygmund kernel} if it satisfies the standard size and continuity conditions: for $x \neq y$
\begin{align*}
\abs{K(x,y)}&\lesssim_K \frac{1}{|x-y|^d},
\intertext{and for ${|x-x'|}\leq \frac{1}{2}|x-y|$}
 \abs{K(x,y)-K(x',y)}+\abs{K(y,x)-K(y,x')}&\leq \frac{1}{|x-y|^d}\, \omega\left( \frac{|x-x'|}{|x-y|}\right).
\end{align*}
The kernel $K:\br^d\times \br^d \setminus \{x=y\}\to \bc$ is called {\it non-degenerate} if for each $x\in \br^d$ and $r>0$ there is $y\in \br^d$ with $|x-y|\gtrsim_K r$ and $$|K(y,x)|\gtrsim_K \frac{1}{r^d}.$$
\end{definition}
The notion of non-degeneracy in Definition \ref{def:czkernel} was introduced by Hyt\"onen in \cite{hytonen2021}, to which we refer for an overview of preceding non-degeneracy assumptions in the literature.

\begin{definition}[$\omega$-Calder\'on--Zygmund operator] Let $p \in (1,\infty)$ and let  $T:L^p\to L^p$ be a bounded linear operator. Then $T$ is called an $\omega$-Calder\'on--Zygmund operator if there is an $\omega$-Calder\'on--Zygmund kernel $K$ such that for every $f\in L^\infty_c$ one has the kernel representation
$$
Tf(x)=\int_{\R^d} K(x,y)f(y)\dy \qquad x\notin \supp(f).
$$
\end{definition}

Estimates involving a Calder\'on--Zygmund operator will implicitly depend on the operator norm of $T$ on $L^p$, the modulus of continuity $\omega$, and the implicit constant in the definition of the kernel.

\subsection{Commutators}
Next, we turn to the definition of commutators between an $\omega$-Calder\'on--Zygmund operator and pointwise multiplication by $b\in L^1_{\loc}$. In general, the commutator
$$
[b,T]f := bT(f)-T(bf)
$$
may be undefined for $f \in L^\infty_c$, since for such $f$ we only have $bf \in L^1$. If $T$ is weak $L^1$-bounded, this is not an issue, which is for example the case when $\omega$ satisfies the Dini condition. In this case we have the following identity:

\begin{lemma}\label{lemma:commutator_duality}
Let $T$ be a $\omega$-Calder\'on--Zygmund operator and $b \in L^1_{\loc}$. Suppose that $T \colon L^1 \to L^{1,\infty}$ is bounded. For $f\in L^\infty_c$ and $g\in L^\infty_c$ with supports separated by a positive distance, we have 
\begin{align*}
\ip{g, [b,T]f }= \ip{b,gTf-fT^*g}.
\end{align*}
\end{lemma}

\begin{proof}
We start by noting that, by the support condition on $f$ and $g$ and the kernel estimates, we have $gTf,fT^*g \in L^\infty_c$, so the right-hand side is well-defined. It remains to show that 
\begin{equation}\label{eq:lasttoprove}
    \ip{g,T(bf)} = \ip{b,fT^*g}.
\end{equation}

Define $h = bf \in L^1$ and note that $T(h) \in L^{1,\infty}$ is well-defined. Moreover  $T(h \ind_{\abs{h}\leq n}) \to T(h)$ in $L^{1,\infty}$ and thus $T(h \ind_{\abs{h}\leq n_k}) \to T(h)$ pointwise a.e. for some subsequence $(n_k)_{k\geq 1}$. Therefore, for $x \notin \supp (f)$, we have by the dominated convergence theorem
$$
T(bf)(x) = \lim_{k \to \infty} \int_{\R^d} K(x,y) h(y) \ind_{\abs{h} \leq n_k} \dd y = \int_{\R^d} K(x,y) b(y) f(y) \dd y,
$$
from which \eqref{eq:lasttoprove} follows by Fubini's theorem.
\end{proof}

Note that 
$$
\ip{b,gTf-fT^*g}= \int_{\R^d} \int_{\R^d}  (b(x)-b(y))K(x,y)f(y)g(x) \dy \dx
$$
for $f\in L^\infty_c$ and $g\in L^\infty_c$ with supports separated by a positive distance, where the integrand is Lebesgue integrable due to the estimates for $\omega$-Calder\'on--Zygmund kernels.

 In our lower bounds for commutators we will avoid well-definedness issues of $[b,T]$ for $b \in L^1_{\loc}$ without assuming the $T:L^1\to L^{1,\infty}$ boundedness. Indeed, as in \cite{hytonen2021}, we will work 
under 
the assumption that we study an operator $U_b$ with off-support kernel representation with $\omega$-Calder\'on--Zygmund kernel,
$$
\ip{g,U_bf}=\int \int (b(y)-b(x))K(y,x)f(x)g(y) \dx \dy,
$$ 
for $f\in L^\infty_c$ and $g\in L^\infty_c$ with supports separated by a positive distance. Assumptions are formulated entirely in terms of the boundedness of $U_b$ and this kernel representation, with no reference to $T$ or $[b,T]$. The prototype of such an operator is of course the commutator $U_b:=[b,T]$ with $T$ an $\omega$-Calder\'on--Zygmund operator with $\omega$ satisfying the Dini condition.

\subsection{Paraproducts}\label{subsection:para}
We end our preliminaries with the introduction of paraproducts. Let $\mc{D}$ be a dyadic lattice and $b \in L^1_{\loc}$
 The difference  of averages $D_Qb$ for $Q \in \mc{D}$ is defined by
$$D_Qb:=\sum_{R\in \mc{D}:\widehat{R}=Q}\angles{b}_{R}1_{R}-\angles{b}_Q1_Q.$$
The difference $D_Qb$ can also be written in terms of Haar projections
$$
D_Qb=\sum_{i \in \{0,1\}^d\setminus \cbrace{0}} \ip{b,h_Q^{i}}h_Q^{i},
$$
where $(h_Q^i)_{i \in \cbrace{0,1}^d}$ are the Haar functions associated with $Q$. 

The paraproduct $\Pi_b$ associated with $b \in L^1_{\loc}$ is formally defined as
$$
\Pi_bf:= \sum_{Q \in \mc{D}} D_Qb \ip{f}_Q
$$
for $f\in L^\infty_c$.
Since the assumption $b\in L^1_{\loc}$ alone is insufficient to make sense of convergence of the sum over all dyadic cubes, we include into the definition an a priori  unconditional convergence assumption: 

\begin{definition}[Paraproduct]\label{def:paraproduct} Let $b \in L^1_{\loc}$. Assume moreover that $b$ is such that for each $f\in L^\infty_c$ the sum
$$
\sum_{Q \in \mc{D}} D_Qb \ip{f}_Q
$$
 converges unconditionally in $L^1_{\loc}$, by which we mean unconditional convergence in $L^1(K)$ for every compact set $K$. Then the operator $\Pi_b:L^\infty_c\to L^1_{\loc}$ defined by
$$
\Pi_bf:=\Pi_{b,\cd}f:=\sum_{Q \in \mc{D}} D_Qb \ip{f}_Q
$$
is called the {\it dyadic paraproduct} associated with the function $b$. 
\end{definition}
We write $\Pi_{b,\cd'}f:= \sum_{Q \in \mc{D'}} D_Qb \ip{f}_Q$ when $\cd'\subseteq \cd$.

Stronger conditions on $b$ ensure stronger unconditional convergence. In the unweighted setting, for example, if $b \in \BMO$, then for every $f\in L^p$ the sum $\Pi_{b,\cd}f$ converges unconditionally in $L^p$ for every $p \in (1,\infty)$. In the extension of this to the Bloom setting, Burkholder's weak $L^1$-inequality is a useful tool:

\begin{lemma}[Burkholder]\label{lemma:burkholder}
Let $\mc{D}$ be a dyadic lattice and let $(v_Q)_{Q \in \mc{D}}$ be a finitely nonzero sequence of scalars. We have for $f \in L^1$  
   \begin{align*}
     \nrms{\sum_{Q \in \mc{D}} v_Q D_Qf}_{L^{1,\infty}} &\leq 2 \sup_{Q \in \mc{D}} \abs{v_Q}\,\nrm{f}_{L^1}.
   \end{align*}
\end{lemma}

\begin{proof}
Viewing $\mc{D}$ as a filtration on $\R^d$, and rescaling such that $\sum_{Q \in \mc{D}} v_Q D_Qf$ is supported on $[-\frac12,\frac12]^d$,
we can view $\sum_{Q \in \mc{D}} v_Q D_Qf$ as a martingale transform. The statement then follows from Burkholder's weak $L^1$-inequality for martingale transforms \cite{Bu79}.
\end{proof}

\begin{remark}
  In $\R^d$, there are two more paraproducts of interest, which can be treated using similar methods as employed in the current paper. We refer to \cite{FH22} for a discussion of these paraproducts in the context of Bloom boundedness with $p=q$.   
\end{remark}

\section{Upper bounds}\label{sec:upper}
In this section we will give sufficient conditions for the Bloom boundedness of commutators and paraproducts  for all $1<p,q<\infty$. The results in the case $q<p$ are entirely new, whereas in the case $p\leq q$ the results for paraproducts are new and for commutators we obtain sharper quantitative bounds in terms of the weight characteristics  than currently available in the literature.

\subsection{Commutators}
We start with the upper bound for commutators, for which we will use the main sparse domination result of \cite{LOR17}. To state their result, for $f,b \in L^1_{\loc}$ and a sparse family of cubes $\mc{S}$, we define the auxiliary sparse operator
\begin{align*}
  \mc{A}_{\mc{S},b}f(x)&=\sum_{Q\in \mc{S}}|b-\ip{b}_Q|\ip{f}_Q\ind_Q(x), &&x \in \R^d,
  \intertext{and its formal adjoint}
  \mc{A}_{\mc{S},b}^\star f(x) &=\sum_{Q\in \mc{S}}\ipb{|b-\ip{b}_Q|f}_Q\ind_Q(x), &&x \in \R^d.
\end{align*}

\begin{theorem}[{\cite[Theorem 1.1]{LOR17}}] \label{theorem:LOR}
Let $T$ be an $\omega$-Calder\'on-Zygmund operator with $\omega$ satisfying the Dini condition and let $b \in L^1_{\loc}$. For every $f\in L^\infty_c$, there exist $3^d$ dyadic lattices $\mc{D}^{k}$ and $\frac{1}{2\cdot 9^d}$-sparse families $\mc{S}_k\subseteq\mc{D}^{k}$ such that
\begin{equation*}
\absb{[b,T]f(x)}\lesssim  \sum_{k=1}^{3^d}\big(\mathcal{A}_{\mc{S}_k,b}|f|(x)+\mathcal{A}_{\mc{S}_k,b}^\star |f|(x)\big), \qquad x \in \R^d.
\end{equation*}
\end{theorem}

In view of Theorem \ref{theorem:LOR}, to prove Bloom upper estimates for the commutators $[b,T]$, it suffices prove Bloom estimates for $\mc{A}_{\mc{S},b}$ and $\mc{A}_{\mc{S},b}^\star$. 
We will need the following lemma, which is a special case of the main result in \cite{FH18}.

\begin{lemma}[\cite{FH18}]\label{lemma:weightedsparse} 
Let $1<p \leq q <\infty$, set $\frac1s := \frac1p+\frac1{q'}$ and let $w \in A_q$. Let $\mc{D}$ be a dyadic lattice. For any $\gamma$-sparse collection $\mc{S} \subseteq \mc{D}$ we have
\begin{align*}
  \nrms{\sum_{Q\in \mc{S}}\has{\frac{1}{\abs{Q}^{s}}\int_Q\abs{f}^{s}\dx}^{1/s}\ind_Q }_{L^p(w^{p/q} ) \to L^q(w)} \lesssim [w]_{A_q}^{\max\cbrace{\frac1{p'}+\frac1q,\frac{1}{q-1}}}
\end{align*}
\end{lemma}

\begin{proof}
  Let $f \in L^p(w)$ and write
  \begin{equation*}
   \mc{A}_{\mc{S},s}f := {\sum_{Q \in \mc{S}} \has{\frac{1}{\abs{Q}^{s}}\int_Q \abs{f}^{s}\dx}^{1/s}  \ind_Q}.
  \end{equation*}
  Using \cite[Theorem 1.1]{FH18} with parameters 
  $p = p/s$, $q= q/s$, $r= 1/s$, $\alpha = s$, $\omega=w$ and $\sigma = w^{-q'/q}=w'$, we obtain
  \begin{align*}
    \nrm{\mc{A}_{\mc{S},s}f}_{L^q(w)} &= \nrmb{(\mc{A}_{\mc{S},s}f)^{s}}_{L^{q/s}(w)}^{1/s}\\
    &\lesssim \has{\sup_{Q \in \mc{S}}\, \abs{Q}^{-s} w(Q)^{\frac{s}q} w'(Q)^{\frac{p-s}{p}} \hab{[w]_{A_\infty}^{\frac{s}{p'}}+[w']_{A_\infty}^{\frac{s}{q}} }}^{1/s}    \nrms{\frac{\abs{f}^{s}}{w'}}_{L^{p/s}(w')}^{1/s}\\
    &\lesssim \sup_{Q \in \mc{S}} \, \abs{Q}^{-1} w(Q)^{1/q} w'(Q)^{\frac{1}{q'}} \hab{[w]_{A_\infty}^{\frac{1}{p'}}+[w']_{A_\infty}^{\frac{1}{q}}  }  \nrm{f}_{L^{p}(w^{p/q})},
    \end{align*}
where we used 
\begin{align*}
    (1-q') \has{1- \frac{p}{s}} &= (1-q')\frac{p}{q'} =\frac{p}{q}\\
    \frac{p-s}{ps} &= \frac{1}{s} - \frac{1}{p} = \frac{1}{q'},
\end{align*}
in the last step. Noting that 
    \begin{equation*}
      \sup_{Q \in \mc{S}} \,\abs{Q}^{-1} w(Q)^{1/q} w'(Q)^{\frac{1}{q'}} \leq [w]_{A_q}^{\frac{1}{q}},
    \end{equation*}
and by Lemma \ref{lemma:muckenhoupt_basics}
    \begin{equation*}
      [w]_{A_\infty}^{\frac{1}{p'}}+[w']_{A_\infty}^{\frac{1}{q}} \leq [w]_{A_q}^{\frac{1}{p'}} + [w]_{A_q}^{\frac{1}{q(q-1)}} \leq 2\, [w]_{A_q}^{\max\cbrace{\frac1{p'},\frac{1}{q-1}-\frac{1}{q}}},
    \end{equation*}
    finishes the proof.
\end{proof}

Using Lemma \ref{lemma:weightedsparse}, we can now prove Bloom estimates for $\mc{A}_{\mc{S},b}^\star$.

\begin{proposition}\label{prop:bsparsebounded}
Let $1<p,q<\infty$ and define $\frac1r:=\frac1q-\frac1p$ and $\frac{\alpha}{d} := \frac1p-\frac1q$. Take
$\mu \in A_p$, $\lambda \in A_q$ and set $\nu^{\frac1p+\frac1{q'}} := \mu^{\frac1p} \lambda^{-\frac1q}$. Let $\mc{D}$ be a dyadic lattice. For any $\gamma$-sparse collection $\mc{S} \subseteq \mc{D}$ and $b \in L^1_{\loc}$ we have
\begin{align*}
  \nrm{\mc{A}_{\mc{S},b}^\star}_{L^p(\mu) \to L^q(\lambda)} \lesssim  [\mu]_{A_p}^{\max\cbrace{1,\frac{1}{p-1}}} \begin{cases} 
   [\lambda]_{A_q}^{\max\cbrace{\frac{1}{p'}+\frac{1}{q},\frac{1}{q-1}}} \cdot\nrm{b}_{\BMO_\nu^\alpha} & p\leq q,\\
      [\lambda]_{A_q}^{\max\cbrace{1,\frac{1}{q-1}}}\cdot \nrm{M^{\sharp}_{\nu}b}_{L^r(\nu)}  \qquad &p>q.
  \end{cases}
\end{align*}
\end{proposition}

\begin{proof}
We start by noting that by \cite[Lemma 5.1]{LOR17}, there is a
$\frac{\gamma}{4}$-sparse collection $\widetilde{\mc{S}} \subset \mc{D}$ with $\mc{S} \subseteq \widetilde{\mc{S}}$ so that for any $Q\in\mc{S}$ we have
\begin{align}\label{eq:secondsparse}\begin{aligned}
  \int_Q\abs{b-\ip{b}_Q}\abs{f}\dx &\lesssim \int_Q\sum_{R\in\widetilde{\mc{S}}: R\subseteq Q}\ipb{|b-\ip{b}_R|}_R |f| \ind_R\dx\\
&= \sum_{R\in\widetilde{\mc{S}}: R\subseteq Q} \int_R|b-\ip{b}_R|\cdot  \ip{|f|}_R\dx.
\end{aligned}
\end{align}
Furthermore, in order to use Lemma \ref{lemma:weightedsparse} efficiently, we define for $s \in (0,\infty)$
\begin{align*}
  \mc{A}_{\mc{S},s}f&:=\sum_{Q\in \mc{S}}\has{\frac{1}{\abs{Q}^{s}}\int_Q\abs{f}^{s}\dx}^{1/s}\ind_Q.
\end{align*}

We first consider the case $p \leq q$. Since $\frac{1}{s}:= \frac1p+\frac1{q'}\geq 1$, we have by  Minkowski's inequality
\begin{align*}
  \sum_{R\in\widetilde{\mc{S}}: R\subseteq Q} \int_R|b-\ip{b}_R|\cdot  \ip{|f|}_R \dd x&\leq \nrm{b}_{\BMO_\nu^\alpha} \sum_{R\in\widetilde{\mc{S}}: R\subseteq Q} \nu(R)^{1/s} \ip{\abs{f}}_R \\
  &= \nrm{b}_{\BMO_\nu^\alpha} \sum_{R\in\widetilde{\mc{S}}: R\subseteq Q}\has{ \int_Q \ip{\abs{f}}_R^{s} \ind_R \dnu}^{1/s} \\
  &\leq \nrm{b}_{\BMO_\nu^\alpha} \has{ \int_Q \has{\sum_{R\in\widetilde{\mc{S}}: R\subseteq Q}\ip{\abs{f}}_R \ind_R }^{s} \dnu}^{1/s}.
\end{align*}
And thus, combined with \eqref{eq:secondsparse}, Lemma \ref{lemma:weightedsparse} and the definition of $\nu$, we obtain for $f \in L^p(\mu)$
\begin{align*}
     \nrm{\mc{A}_{\mc{S},b}^\star f}_{L^q(\lambda)} &\lesssim \nrm{b}_{\BMO_\nu^\alpha} \nrmb{\mc{A}_{\mc{S},s}\hab{\mc{A}_{\widetilde{\mc{S}},1}f \cdot \nu^{1/s}}}_{L^q(\lambda)} \\
     &\lesssim \nrm{b}_{\BMO_\nu^\alpha} [\lambda]_{A_q}^{\max\cbrace{\frac{1}{p'}+\frac{1}{q},\frac{1}{q-1}}} \nrmb{\mc{A}_{\widetilde{\mc{S}},1}f \cdot \nu^{1/s}}_{L^p(\lambda^{p/q})}\\
     &=\nrm{b}_{\BMO_\nu^\alpha} [\lambda]_{A_q}^{\max\cbrace{\frac{1}{p'}+\frac{1}{q},\frac{1}{q-1}}} \nrm{\mc{A}_{\widetilde{\mc{S}},1}f}_{L^p(\mu)}\\
     &\lesssim \nrm{b}_{\BMO_\nu^\alpha} [\lambda]_{A_q}^{\max\cbrace{\frac{1}{p'}+\frac{1}{q},\frac{1}{q-1}}} [\mu]_{A_p}^{\max\cbrace{1,\frac{1}{p-1}}} \nrm{f}_{L^p(\mu)},
  \end{align*}
finishing the case $p\leq q$.

  Next, we consider the case $p>q$. We have
\begin{align*}
 \sum_{R\in\widetilde{\mc{S}}: R\subseteq Q} \int_R|b-\ip{b}_R|\cdot  \ip{|f|}_R \dx &= \int_Q \sum_{R\in\widetilde{\mc{S}}:R\subseteq Q}\frac{1}{\nu(R)}\int_R{|b-\ip{b}_R|} \dd x \cdot \ip{|f|}_R \ind_R \dnu \\
  &\leq \int_Q M^{\sharp}_{\nu}b \cdot \mc{A}_{\widetilde{\mc{S}},1}f  \dnu.
\end{align*}
  Thus, combined with \eqref{eq:secondsparse} and Lemma \ref{lemma:weightedsparse}, we obtain for $f \in L^p(\mu)$
  \begin{align*}
     \nrm{\mc{A}_{\mc{S},b}^\star f}_{L^q(\lambda)} &\lesssim \nrmb{\mc{A}_{\mc{S},1}\hab{ M^{\sharp}_{\nu}b \cdot \mc{A}_{\widetilde{\mc{S}},1}f \cdot \nu}}_{L^q(\lambda)} \\
     &\lesssim [\lambda]_{A_q}^{\max\cbrace{1,\frac{1}{q-1}}} \nrmb{ M^{\sharp}_{\nu}b \cdot \mc{A}_{\widetilde{\mc{S}},1}f \cdot \nu}_{L^q(\lambda)}.
  \end{align*}
  Finally, using H\"older's inequality, the definition of $\nu$ and Lemma \ref{lemma:weightedsparse} once more, we obtain
  \begin{align*}
    \nrmb{ M^{\sharp}_{\nu}b \cdot \mc{A}_{\widetilde{\mc{S}},1}f \cdot \nu}_{L^q(\lambda)} &= \nrmb{ M^{\sharp}_{\nu}b \cdot \nu^{1/r} \cdot \mc{A}_{\widetilde{\mc{S}},1}f \cdot \nu^{1/r'} \lambda^{1/q}}_{L^q}\\
    &\leq \nrmb{M^{\sharp}_{\nu}b \cdot \nu^{1/r}}_{L^r} \,\nrmb{\mc{A}_{\widetilde{\mc{S}},1}f \cdot \nu^{1/r'} \lambda^{1/q}}_{L^p}\\
    &= \nrm{M^{\sharp}_{\nu}b}_{L^r(\nu)} \,\nrm{\mc{A}_{\widetilde{\mc{S}},1}f}_{L^p(\mu)}\\
    &\lesssim [\mu]_{A_p}^{\max\cbrace{1,\frac{1}{p-1}}} \nrm{M^{\sharp}_{\nu}b}_{L^r(\nu)} \,\nrm{f}_{L^p(\mu)},
  \end{align*}
  finishing the proof.
\end{proof}

Since $\mc{A}_{\mc{S},b}^\star$ is the formal adjoint of $\mc{A}_{\mc{S},b}$, we can also deduce upper Bloom estimates for $\mc{A}_{\mc{S},b}$ from Proposition \ref{prop:bsparsebounded}.
Combining Theorem \ref{theorem:LOR} and Proposition \ref{prop:bsparsebounded}, we therefore obtain our desired result. 
\begin{theorem}[Bloom upper estimate for commutators]\label{theorem:uppercommutator}
Let $p,q \in (1,\infty)$ and define $\frac1r:=\frac1q-\frac1p$ and $\frac{\alpha}{d} := \frac1p-\frac1q$. Take
$\mu \in A_p$, $\lambda \in A_q$ and set $\nu^{\frac1p+\frac1{q'}} := \mu^{\frac1p} \lambda^{-\frac1q}$. Let $T$ be an $\omega$-Calder\'on-Zygmund operator with $\omega$ satisfying the Dini condition and let $b \in L^1_{\loc}$. We have
  \begin{align*}
  \nrmb{[b,T]}_{L^p(\mu) \to L^q(\lambda)} \lesssim  [\mu]_{A_p}^{\max\cbrace{1,\frac{1}{p-1}}} [\lambda]_{A_q}^{\max\cbrace{1,\frac{1}{q-1}}}\cdot \begin{cases} \nrm{b}_{\BMO_\nu^\alpha} & p\leq q,\\
      \nrm{M^{\sharp}_{\nu}b}_{L^r(\nu)}  \qquad &p>q.
  \end{cases}
\end{align*}
\end{theorem}

\begin{proof}
  By Proposition \ref{prop:bsparsebounded} and duality, using $\frac{1}{p'} +\frac1q \leq 1$ when $p\leq q$, we have
  \begin{align*}
  \nrm{\mc{A}_{\mc{S},b}}_{L^p(\mu) \to L^q(\lambda)} &= \nrm{\mc{A}_{\mc{S},b}^\star}_{L^{q'}(\lambda') \to L^{p'}(\mu')} \\&\lesssim [\mu']_{A_{p'}}^{\max\cbrace{1,\frac{1}{p'-1}}} [\lambda']_{A_{q'}}^{\max\cbrace{1,\frac{1}{q'-1}}} \cdot\begin{cases} \nrm{b}_{\BMO_\nu^\alpha} & p\leq q\\
      \nrm{M^{\sharp}_{\nu}b}_{L^r(\nu)}  \qquad &p>q
  \end{cases}\\
  &= [\mu]_{A_p}^{\max\cbrace{1,\frac{1}{p-1}}} [\lambda]_{A_q}^{\max\cbrace{1,\frac{1}{q-1}}}\cdot \begin{cases} \nrm{b}_{\BMO_\nu^\alpha} & p\leq q\\
      \nrm{M^{\sharp}_{\nu}b}_{L^r(\nu)}  \qquad &p>q.
  \end{cases}
\end{align*}
Therefore, the theorem follows from Theorem \ref{theorem:LOR} and the density of the bounded, compactly supported functions in $L^p(\mu)$.
\end{proof}

\begin{remark}
For $p=q$, Theorem \ref{theorem:uppercommutator} was already obtained in \cite{HLW17}. We generalized the proof that uses sparse domination from \cite{LOR17}, where the case $p=q$ is considered as well. If $\mu=\lambda=w$ for $w \in A_p$, it is known that the dependence on $[w]_{A_p}$ in Theorem \ref{theorem:uppercommutator} is sharp, see \cite{Ch11,Pe19} and the references therein.

A qualitative version of Theorem \ref{theorem:uppercommutator} for $p\leq q$ was recently obtained in \cite{HOS23}. Tracking the constants in \cite[Theorem 2.4]{HOS23} would yield quantitatively worse behavior in $[\mu]_{A_p}$ and $[\lambda]_{A_q}$ than Theorem \ref{theorem:uppercommutator}. Note that the dependence on $[\mu]_{A_p}$ and $[\lambda]_{A_q}$ in Theorem \ref{theorem:uppercommutator} can be slightly improved in the case $p\leq q$, using the full power of Proposition \ref{prop:bsparsebounded}. Since we do not know if the obtained bound is sharp, we leave this to the interested reader.
\end{remark}

\subsection{Paraproducts}
Next, we consider Bloom upper bounds for paraproducts.
 We start our analysis with a sparse domination result for finite truncations of $\Pi_b$ for $b \in L^1_{\loc}$, which generalizes \cite[Theorem 4.1]{FH22}. We will employ a stopping time argument that has two innovative features: ``coupled stopping conditions'' and "uniformity over parts of the input data".  In our further considerations it will be crucial that this sparse domination result still contains the terms
 $\ipb{\abs{b-\ip{b}_{Q}}}_{Q}$
 rather than the more typical sparse estimate of paraproducts using $\nrm{b}_{\BMO}$.

\begin{theorem}[Sparse domination of paraproducts]\label{theorem:sparsepara}
Let $f,b \in L^1_{\loc}(\R^d)$, and let $\mc{D}$ be a dyadic lattice. For every $Q_0 \in \mc{D}$, there exist a $\frac{1}{2^{d+2}}$-sparse family $\mc{S}\subseteq\mc{D}(Q_0)$ such that for all finite collections $\mc{F} \subseteq \mc{D}(Q_0)$
\begin{equation*}
  \abss{ \sum_{Q \in \mc{F}} D_Qb \ip{f}_Q}\lesssim \sum_{Q \in \mc{S}}\ipb{\abs{b-\ip{b}_{Q}}}_{Q} \ip{\abs{f}}_{Q}\ind_Q.
\end{equation*}
\end{theorem}

\begin{proof}
  We will show that for each $Q \in \mc{D}$ there exists a collection of pairwise disjoint cubes $\cbrace{P_k}_k\subseteq \mc{D}(Q)$ such that $\sum_{k} \abs{P_k} \leq \frac12 \abs{Q}$ and such that
  \begin{align}\label{eq:sparse1steppara}
  \begin{aligned}
      \abss{ \sum_{R \in \mc{D}(Q)\cap \mc{F}} D_Rb  \ip{f}_R} &\leq 2^6 \ipb{\abs{b-\ip{b}_{Q}}}_{Q} \ip{\abs{f}}_{Q} \ind_{Q}
      \\&\hspace{0.2cm} +  \sum_{k} \abs{D_{\widehat{P}_k}b} \ip{\abs{f}}_{\widehat{P}_k} \ind_{P_k} + \sum_{k} \abss{\sum_{R \in \mc{D}(P_k)\cap \mc{F}} D_Rb\ip{f}_R}
  \end{aligned}
  \end{align}
   uniformly over all finite collections $\mc{F}\subseteq \mc{D}$. Note that $$\abs{D_{\widehat{P}_k}b}\ind_{P_k}= \abs{\ip{b}_{P_k}-\ip{b}_{\widehat{P}_k}} \ind_{P_k}.$$ Iterating the estimate \eqref{eq:sparse1steppara}, starting at the cube $Q_0$, we obtain a $\frac12$-sparse collection $\tilde{\mc{S}}\subseteq \mc{D}(Q_0)$ such that
\begin{equation}\label{eq:extrainteresting}
    \begin{aligned}
         \abss{ \sum_{Q \in \mc{F}} D_Qb \ip{f}_Q}&\leq  \sum_{Q \in \tilde{\mc{S}}}2^6 \ipb{\abs{b-\ip{b}_{Q}}}_{Q} \ip{\abs{f}}_{Q} \ind_Q \\&
         \hspace{1cm} +\sum_{Q \in \tilde{\mc{S}} \setminus \cbrace{Q_0}}\abs{\ip{b}_Q-\ip{b}_{\widehat{Q}}} \ip{\abs{f}}_{\widehat{Q}}  \ind_Q
    \end{aligned}
\end{equation}
   uniformly over all finite collections $\mc{F}\subseteq \mc{D}(Q_0)$. Noting that
  $$
  \mc{S}:= \tilde{\mc{S}} \cup \cbraceb{\widehat{Q}: Q \in \tilde{\mc{S}}\setminus \cbrace{Q_0}} \subseteq \mc{D}(Q_0) 
  $$
is $\frac{1}{2^{d+2}}$-sparse and for every $Q \in \mc{D}$ we have
$$
\abs{\ip{b}_Q-\ip{b}_{\widehat{Q}}} \ip{\abs{f}}_{\widehat{Q}}\ind_Q\leq 2^d \ipb{\abs{b-\ip{b}_{\widehat{Q}}}}_{\widehat{Q}} \ip{\abs{f}}_{\widehat{Q}} \ind_{\widehat{Q}}
$$
then yields the result.

\medskip

Let us prove \eqref{eq:sparse1steppara}.   Let $\cbrace{P_k}_k$ be the collection of maximal cubes  $P \in \mc{D}(Q)$ such that 
  \begin{align}
    \label{eq:stopping1} \ip{\abs{f}}_P &> 4 \, \ip{\abs{f}}_Q
    \intertext{ or there exists a finite collection $\mc{F}_P$ such that on  $P$ we have }
   \label{eq:stopping2} \abss{\sum_{R \in \mc{F}_P: P\subsetneq R\subseteq Q} D_Rb \ip{f}_R} &> 2^5 \, \ipb{\abs{b-\ip{b}_{Q}}}_{Q}\ip{\abs{f}}_{Q} =:a_Q.
  \end{align}
We remark that the left-hand side is constant on $P$.
  
  For $R \in \mc{D}(Q)$ define
$$
v_R:=\begin{cases}
\ip{f}_R & R \in \bigcup_{k} \cbrace{S \in \mc{F}_{P_k}: P_k \subsetneq  S\subseteq Q} \\
0 & \text{otherwise}.
\end{cases}
$$
By the stopping condition \eqref{eq:stopping1} we have $\abs{v_R}\leq 4 \ip{\abs{f}}_Q$ for $R \in \mc{D}(Q)$. 
 Therefore, using the stopping conditions \eqref{eq:stopping1} and \eqref{eq:stopping2}, Burkholder's weak $L^1$-inequality for martingale differences (see Lemma \ref{lemma:burkholder}) and the weak $L^1$-boundedness of the dyadic maximal operator, we have
\begin{align*}
  \sum_{k} \abs{P_k} &\leq \abss{ \cbraces{ \absb{\sum_{R \in \mc{D}(Q)} v_R D_Rb}> a_Q }}+\abss{\cbraces{\sup_{R \in \mc{D}(Q)} \ip{\abs{f}}_R > 4\,\ip{\abs{f}}_Q}} \\
  &\leq \frac{2}{a_Q}  \cdot \sup_{R \in \mc{D}(Q)} \abs{v_R} \cdot  \nrms{{\sum_{R \in \mc{D}(Q)} D_Rb}}_{L^1} + \frac{1}{4\, \ip{\abs{f}}_Q} \cdot \nrm{f \ind_Q}_{L^1}\\
  &\leq \frac{2}{2^5 \ipb{\abs{b-\ip{b}_{Q}}}_{Q}\ip{\abs{f}}_{Q}}  \cdot 4\ip{\abs{f}}_Q \cdot  \nrm{b-\ip{b}_Q}_{L^1} + \frac{\abs{Q}}{4}\\
  &= \tfrac12\abs{Q}.
\end{align*}
 Now, to show \eqref{eq:sparse1steppara}, let $\mc{F}\subseteq \mc{D}$ be a finite collection of cubes and write
\begin{align*}
    \abss{ \sum_{R \in \mc{D}(Q)\cap \mc{F}} D_Rb  \ip{f}_R}
    &\leq \abss{ \sum_{R \in \mc{D}(Q)\cap \mc{F}} D_Rb \ip{f}_R} \ind_{Q \setminus \bigcup_k P_k}\\&\hspace{1cm}+\sum_{k} \abss{\sum_{R \in \mc{F}:\widehat{P}_k\subsetneq R\subseteq Q} D_Rb  \angles{f}_R} \ind_{P_k}
    \\&\hspace{1cm}+\sum_{k} \abs{D_{\widehat{P}_k}b} \ip{\abs{f}}_{\widehat{P}_k} \ind_{P_k}
    \\&\hspace{1cm}+
\sum_{k} \abss{\sum_{R \in \mc{D}(P_k)\cap \mc{F}} D_Rb \angles{f}_R}.
\end{align*}
By the stopping condition \eqref{eq:stopping2}, we have 
$$
 \ind_{Q \setminus \bigcup_k P_k}\abss{ \sum_{R \in \mc{D}(Q)\cap \mc{F}} D_Rb \ip{f}_R}\leq \,2^5 \ipb{\abs{b-\ip{b}_{Q}}}_{Q} \ip{\abs{f}}_{Q},
$$
and
$$
\sum_{k} \ind_{P_k}\abss{\sum_{R \in \mc{F}:\widehat{P}_k\subsetneq R\subseteq Q} D_Rb \angles{f}_R}\leq 2^5\, \ipb{\abs{b-\ip{b}_{Q}}}_{Q} \ip{\abs{f}}_{Q},
$$
both uniformly over all finite collections $\cf\subseteq \cd$.
This concludes the proof of \eqref{eq:sparse1steppara} and thus finishes the proof of the theorem.
\end{proof}

\begin{remark}
    The estimate in \eqref{eq:extrainteresting} is interesting in its own right, since it does not use the doubling property of the Lebesgue measure. In fact, it directly generalizes to the setting where one has a locally finite Borel measure $\mu$ on $\R^d$ instead of the Lebesgue measure.
\end{remark}

Let $\mu, \lambda, \nu$ be weights and $f, b \in L^1_{\loc}$. In the case  $1<p \leq q<\infty$  the sparse operator arising in Theorem \ref{theorem:sparsepara} can be estimated as follows
\begin{equation*}
  \nrms{\sum_{Q \in \mc{S}}\ipb{\abs{b-\ip{b}_{Q}}}_{Q} \ip{\abs{f}}_{Q}\ind_Q}_{L^q(\lambda)} \leq \nrms{\sum_{Q \in \mc{S}}\frac{\nu(Q)^{^{\frac{1}{p}+\frac1{q'}}}}{\abs{Q}} \ip{\abs{f}}_{Q}\ind_Q}_{L^q(\lambda)} \nrm{b}_{\BMO_\nu^\alpha}.
\end{equation*}
Therefore, to prove the Bloom-weighted boundedness of the (finitely truncated) paraproduct $\Pi_{b,\mc{F}}f$, one could analyze the Bloom-weighted boundedness of the sparse operator
$$
f \mapsto \sum_{Q \in \mc{S}}\frac{\nu(Q)^{\frac{1}{p}+\frac1{q'}}}{\abs{Q}} \ip{\abs{f}}_{Q} \ind_Q.
$$
In the case $p=q$ and $\mu,\lambda \in A_p$, this operator was studied in \cite[Theorem 3.9]{FH22}. 
Note that, viewing a paraproduct as a bilinear operator, one  expects weighted boundedness to hold for $(\mu,\lambda)$ in a genuinely multilinear weight class (cf. \cite[Section 4.6]{ALM22}). 

The next lemma identifies the canonical weight class for Bloom estimates of paraproducts for all $p,q \in (1,\infty)$. In Proposition \ref{prop:weightedparaproduct}, we will study Bloom estimates for the sparse operator  arising in Theorem \ref{theorem:sparsepara} using this weight class.
\begin{lemma}[Bloom weight class for paraproducts]\label{lem:weightnec}
Let $1<p,q<\infty$ and define $\frac1r:=\frac1q-\frac1p$ and $\frac{\alpha}{d} := \frac1p-\frac1q$. Let $\mu,\lambda,\nu$ be weights and $\mu' = \mu^{-p'/p}$. Suppose that for all $b \in L^1_{\loc}$ we have 
\begin{align*}
        \nrm{\Pi_b}_{L^p(\mu) \to L^q(\lambda)} \lesssim \begin{cases} \nrm{b}_{\BMO_\nu^\alpha} & p\leq q,\\
      \nrm{M^{\sharp}_{\nu}b}_{L^r(\nu)}  \quad &p>q.
  \end{cases}
    \end{align*} 
Then we have $(\mu',\lambda,\nu) \in B_{p',q}$.
\end{lemma}
\begin{proof}
For a fixed cube $Q \in \mc{D}$, define $b = \ind_{Q_+} - \ind_{Q_-}$, where $Q_+$ and $Q_-$ are the right and left halves of $Q$ along the first coordinate axis respectively. Note that
\begin{align*}
\nrm{b}_{\BMO_\nu^\alpha} &= \frac{\abs{Q}}{\nu(Q)^{\frac1p+\frac1{q'}}} && p\leq q,\\
      \nrm{M^{\sharp}_{\nu}b}_{L^r(\nu)} &=\frac{\abs{Q}}{\nu(Q)} \nrm{M^{\nu}\ind_Q}_{L^r(\nu)} \leq r' \frac{\abs{Q}}{\nu(Q)^{\frac1p+\frac1{q'}}} \quad &&p>q.
\end{align*}
Applying
the boundedness of $\Pi_b$ to $f  := \ind_Q \mu'$, we get
\begin{align*}
 \ip{\mu'}_Q \cdot \nrm{\ind_Q}_{L^q(\lambda)} \lesssim \frac{\abs{Q}}{\nu(Q)^{\frac{1}{p}+\frac1{q'}}}  \cdot \has{\int_Q \mu^{-\frac{p}{p-1}} \dd\mu}^{\frac1p}.
\end{align*}
Rearranging the terms, we obtain
$$
\ip{\mu'}_Q^{\frac{1}{p'}} \ip{\lambda}^{\frac1q}_Q  \ip{\nu}^{\frac{1}{p}+\frac1{q'}}_Q \leq C,
$$
for some $C>0$ independent of $Q$.
\end{proof}

By Lemma \ref{lem:weightnec} and Remark \ref{rem:Bloomsettinginduced} it is canonical to assume the   Bloom relation $\nu^{\frac1p+\frac1{q'}} := \mu^{\frac1p} \lambda^{-\frac1q}$ when studying upper bounds for paraproducts. In this setting, the assumption $(\mu',\lambda,\nu) \in B_{p',q}$ is exactly the same as the assumption that $(\mu,\lambda') \in A_{(p,q')}$, where $A_{(p,q')}$ denotes a multilinear weight class (see \cite{LOPTT09}). This assumption is strictly weaker than $\mu \in A_p$ and $\lambda \in A_q$ (see Lemma \ref{lemma:muckenhoupt_class_bloom}).

In the following proposition, we will prove the Bloom estimates for the sparse operator arising in Theorem \ref{theorem:sparsepara} for weights 
$(\mu',\lambda,\nu) \in B_{p',q}$ satisfying the Bloom relation $\nu^{\frac1p+\frac1{q'}} := \mu^{\frac1p} \lambda^{-\frac1q}$. Note that  $\mu', \lambda,\nu \in A_\infty$ by Lemma \ref{lemma:muckenhoupt_class_bloom}.

\begin{proposition}\label{prop:weightedparaproduct}
Let $1<p,q<\infty$ and define $\frac1r:=\frac1q-\frac1p$ and $\frac{\alpha}{d} := \frac1p-\frac1q$.
Let $(\mu',\lambda,\nu) \in B_{p',q}$ with $\mu'=\mu^{-p'/p}$ and $\nu^{\frac1p+\frac1{q'}} = \mu^{\frac1p} \lambda^{-\frac1q}$.
Let $\mc{D}$ be a dyadic lattice and $\gamma \in (0,1)$. For any $\gamma$-sparse family $\mc{S} \subseteq \mc{D}$ and $b \in L^1_{\loc}$ we have
  \begin{align*}
    \nrms{f \mapsto \sum_{Q \in \mc{S}}\ipb{\abs{b-\ip{b}_{Q}}}_{Q} \ip{\abs{f}}_{Q} \ind_Q}_{L^p(\mu) \to L^q(\lambda)} \lesssim C_{\mu,\lambda,\nu}  \cdot \begin{cases} \nrm{b}_{\BMO_\nu^\alpha} & p\leq q,\\
      \nrm{M^{\sharp}_{\nu}b}_{L^r(\nu)}  \quad &p>q,
  \end{cases}
  \end{align*}
 where
 $$
 C_{\mu,\lambda} = [(\mu',\lambda,\nu)]_{B_{p',q}} \cdot [\mu']_{A_\infty}^{1/p} [\lambda]_{A_\infty}^{1/q'} [\nu]_{A_\infty}^{(1/r)_+}.
 $$
\end{proposition}

\begin{proof}
Throughout the proof we will write  $r_+ = \frac{1}{(1/r)_+}$. Note that
 $$
  \nrm{h}_{L^p(\mu)} = \nrm{h / {\mu'}}_{L^p(\mu')}, \qquad h \in L^p(\mu),
  $$
so, by duality, it suffices to show
\begin{align*}
 \sum_{Q \in \mc{S}}& \int_Q \abs{b-\ip{b}_Q}  \ipb{\abs{f}\mu' }_{Q} \ipb{\abs{g} \lambda}_Q \dx \lesssim C_{\mu,\lambda} \nrm{f}_{L^p(\mu')} \nrm{g}_{L^{q'}(\lambda)}\cdot \begin{cases} \nrm{b}_{\BMO_\nu^\alpha} & p\leq q,\\
      \nrm{M^{\sharp}_{\nu}b}_{L^r(\nu)}   &p>q,
      \end{cases}
\end{align*}
for $f \in L^p(\mu')$ and $g \in L^{q'}(\lambda)$. Since $\frac{1}{p}+\frac{1}{q'} + \frac{1}{r_+}\geq 1$, we have by H\"older's inequality and Lemma \ref{lemma:dyadiccarlesonweight}
    \begin{align*}
    \sum_{Q \in \mc{S}}& \int_Q \abs{b-\ip{b}_Q} \cdot \ipb{\abs{f}\mu' }_{Q} \ipb{\abs{g} \lambda}_Q  \dd x\\
    &=     \sum_{Q \in \mc{S}}\frac{1}{\nu(Q)^{\frac{1}{p}+\frac{1}{q'}}} \int_Q \abs{b-\ip{b}_Q} \cdot \ip{\abs{f}}_{Q}^{\mu'} \ip{\abs{g}}_Q^{\lambda}  \ip{\mu'}_Q \ip{\lambda}_Q \cdot \nu(Q)^{{\frac{1}{p}+\frac{1}{q'}}}\dd x\\ 
    &\leq  [(\mu',\lambda,\nu)]_{B_{p',q}} \sum_{Q \in \mc{S}} \frac{1}{\nu(Q)^{\frac{1}{p}+\frac{1}{q'}}} \int_Q \abs{b-\ip{b}_Q} \cdot \ip{\abs{f}}_{Q}^{\mu'} \mu'(Q)^{\frac{1}{p}  }\cdot \ip{\abs{g}}_Q^{\lambda}  \lambda(Q)^{\frac{1}{q'}}\dd x\\
    &\leq[(\mu',\lambda,\nu)]_{B_{p',q}} \has{\sum_{Q \in \mc{S}} \hab{\ip{\abs{f}}_{Q}^{\mu'} }^p \mu'(Q)}^{1/p} \has{\sum_{Q \in \mc{S}} \hab{\ip{\abs{g}}_Q^{\lambda}}^{q'}  \lambda(Q)}^{1/q'} \\&\hspace{1cm} \cdot \has{\sum_{Q \in \mc{S}} \has{\frac{1}{\nu(Q)^{\frac{1}{p}+\frac{1}{q'}}} \int_Q \abs{b-\ip{b}_Q}\dd x}^{r_+}}^{1/r_+}\\ 
    &\leq \frac{pq'}{\gamma^{\frac{1}{p}+\frac{1}{q'}}} \, [(\mu',\lambda,\nu)]_{p,q} [\mu']_{A_\infty}^{1/p} [\lambda]_{A_\infty}^{1/q'} \nrm{f}_{L^p(\mu')} \nrm{g}_{L^{q'}(\lambda)}\\&\hspace{2cm} \cdot \has{\sum_{Q \in \mc{S}}\has{\frac{1}{\nu(Q)^{1/r'}} \int_Q \abs{b-\ip{b}_Q}\dd x}^{r_+}}^{1/r_+}.
  \end{align*}
Now if $p \leq q$, we have $r_+ = \infty$, in which case we see directly that
$$
\has{\sum_{Q \in \mc{S}}\has{\frac{1}{\nu(Q)^{1/r'}} \int_Q \abs{b-\ip{b}_Q}\dd x}^{r_+}}^{1/r_+} \leq  \nrm{b}_{\BMO_\nu^\alpha}.
$$
For the case that $p>q$, we have $r_+=r$ and by Corollary \ref{cor:Ainftysparse}
\begin{align*}
  \has{\sum_{Q \in \mc{S}}&\has{\frac{1}{\nu(Q)^{1/r'}} \int_Q \abs{b-\ip{b}_Q}\dd x}^{r_+}}^{1/r_+}\\
  &\leq\frac{[\nu]_{A_\infty}^{1/r}}{\gamma^{1/r}} \has{\sum_{Q \in \mc{S}}\has{\frac{1}{\nu(Q)} \int_Q \abs{b-\ip{b}_Q}\dd x}^{r}\nu(E_Q)}^{1/r}\\
  &\leq \frac{[\nu]_{A_\infty}^{1/r}}{\gamma^{1/r}} \has{\sum_{Q \in \mc{S}} \int_{E_Q} (M^{\sharp}_\nu b)^r\dd \nu}^{1/r}
  \leq \frac{[\nu]_{A_\infty}^{1/r}}{\gamma^{1/r}} \nrm{M^{\sharp}_\nu b}_{L^r(\nu)},
\end{align*}
finishing the proof.
\end{proof}

Combining Theorem \ref{theorem:sparsepara} and Proposition \ref{prop:weightedparaproduct}, we now obtain a Bloom upper bound for paraproducts for weights  $(\mu',\lambda,\nu) \in B_{p',q}$ with $\nu^{\frac1p+\frac1{q'}} = \mu^{\frac1p} \lambda^{-\frac1q}$. The case $p=q$, $\mu \in A_p$ and $\lambda \in A_q$  has previously been obtained in \cite{FH22}.

\begin{theorem}[Bloom upper estimate for paraproducts]\label{theorem:bloompara}
Let $1<p,q<\infty$ and define $\frac1r:=\frac1q-\frac1p$ and $\frac{\alpha}{d} := \frac1p-\frac1q$.
Let $(\mu',\lambda,\nu) \in B_{p',q}$ with  $\mu' = \mu^{-p'/p}$ and $\nu^{\frac1p+\frac1{q'}} = \mu^{\frac1p} \lambda^{-\frac1q}$. Let $\mc{D}$ be a dyadic lattice. For any $b \in L^1_{\loc}$ we have
  \begin{align*}
    \nrm{\Pi_b}_{L^p(\mu) \to L^q(\lambda)} \lesssim [(\mu',\lambda,\nu)]_{B_{p',q}} [\mu']_{A_\infty}^{1/p} [\lambda]_{A_\infty}^{1/q'}  \cdot \begin{cases} \nrm{b}_{\BMO_\nu^\alpha} & p\leq q,\\
     [\nu]_{A_\infty}^{1/r} \nrm{M^{\sharp}_{\nu}b}_{L^r(\nu)}  \quad &p>q,
  \end{cases}
  \end{align*}
where
$$
\Pi_bf:= \sum_{\substack{Q \in \mc{D}}} D_Qb \ip{f}_Q, \qquad f \in L^p(\mu),
$$
converges unconditionally in $L^q(\lambda)$  if $b \in \BMO_\nu^\alpha$ when $p \leq q$ and if  $M^{\sharp}_{\nu}b \in L^r(\nu)$ when $p>q$.
\end{theorem}

\begin{proof}
  Let $f \in L^p(\mu)$. By density we may assume without loss of generality that $f$ has compact support. Moreover, by considering quadrants separately if needed, we may assume that  there is a $Q_0\in \mc{D}$ such that $\supp{f} \subseteq Q_0$.
  By Theorem \ref{theorem:sparsepara}, we can find a $\frac{1}{2^{d+2}}$-sparse family of cubes $\mc{S} \subseteq \mc{D}(Q_0)$ such that for all finite collections $\mc{F}\subseteq \mc{D}(Q_0)$ we have
\begin{equation*}
  \abss{ \sum_{ Q\in \mc{F}} D_Qb \ip{f}_Q}\lesssim \sum_{Q \in \mc{S}}\ipb{\abs{b-\ip{b}_{Q}}}_{Q} \ip{\abs{f}}_{Q} \ind_Q.
\end{equation*}
Now set $\widetilde{\mc{S}} = \mc{S} \cup \cbrace{Q \in \mc{D}:Q_0 \subsetneq Q}$, which is $\frac{1}{2^{d+2}}$-sparse as well. Since, for any $Q \in \mc{D}$, we have
  \begin{align*}
    \abs{D_Qb} &= \abss{\sum_{P \in \mc{D}:\widehat{P}=Q }\ip{b}_{P}\ind_{P} - \ip{b}_Q \ind_Q}\\
    &\leq \sum_{P \in \mc{D}:\widehat{P}=Q }\ipb{\abs{b-\ip{b}_Q}}_{P}\ind_{P}
    \leq2^d \ipb{\abs{b-\ip{b}_Q}}_{Q} \ind_Q,
  \end{align*}
  it follows that for all finite collections $\mc{F}\subseteq \mc{D}$ we have
  \begin{align*}
  \abss{\sum_{Q \in  \mc{F}} D_Qb \ip{f}_Q} &= \abss{\sum_{Q \in\mc{F}\cap \mc{D}(Q_0)} D_Qb \ip{f}_Q} + \sum_{Q \in \mc{F}: Q_0\subsetneq Q} \abs{D_Qb} \ip{\abs{f}}_Q\\&\lesssim \sum_{Q \in \widetilde{\mc{S}}}\ipb{\abs{b-\ip{b}_{Q}}}_{Q} \ip{\abs{f}}_{Q} \ind_Q.
\end{align*}
Since the right-hand side is finite a.e. and does not depend on $\mc{F}$, this implies that
$
 \sum_{Q \in \mc{D}} \abs{D_Qb \ip{f}_Q}
$
 converges pointwise a.e. and for any choice of $\epsilon_Q \in \cbrace{0,1}$ we have
$$
\abss{\sum_{Q \in \mc{D}} \epsilon_Q D_Qb \ip{f}_Q }\lesssim\sum_{Q \in \widetilde{\mc{S}}}\ipb{\abs{b-\ip{b}_{Q}}}_{Q} \ip{\abs{f}}_{Q} \ind_Q.
$$
  The norm estimate now follows from Proposition \ref{prop:weightedparaproduct} and the unconditional convergence follows from the dominated convergence theorem.
\end{proof}

\begin{remark}\label{rem:sharpFH} Let $p,q\in (1,\infty)$, $\mu \in A_p$ and $\lambda \in A_p$ and set $\nu^{\frac1p+\frac1{q'}} := \mu^{\frac1p} \lambda^{-\frac1q}$. By Lemma \ref{lemma:muckenhoupt_basics} and Lemma \ref{lemma:muckenhoupt_implies_bloom} we have
$$
[(\mu',\lambda,\nu)]_{B_{p',q}} [\mu']_{A_\infty}^{1/p} [\lambda]_{A_\infty}^{1/q'} \lesssim [\mu]^{\frac{1}{p-1}}_{A_p} [\lambda]_{A_q},
$$
so, in particular, Theorem \ref{theorem:bloompara} proves the upper bound in Theorem \ref{thm:mainparaproduct}.

Let $w \in A_2$. In the case $p=q=2$, $\mu=w$, $\lambda = w^{-1}$ and thus $\nu = w$, it was shown in \cite[Section 4.1]{FH22} that there is a $b  \in \BMO_w$ such that
  $$
  \nrm{\Pi_b}_{L^2(w) \to L^2(w^{-1})} \gtrsim \nrm{b}_{\BMO_w} [w]_{A_2}^2
  $$
  In this setting we have
  $$
  [\mu]_{A_2} [\lambda]_{A_2} = [w]_{A_2}^2,
  $$
  so our result, as well as \cite[Theorem 3.9]{FH22}, is sharp. We do not know if Theorem \ref{theorem:bloompara} is also sharp for other parameters.
  \end{remark}

\begin{remark}
We note that the claimed unconditional convergence in Theorem \ref{theorem:bloompara} is not automatic, since we do not have $\lambda \in A_q$ in general (see Lemma \ref{lemma:muckenhoupt_class_bloom}). As a consequence, we note that the functions with finite Haar expansion are not necessarily dense in $L^q(\lambda)$.
\end{remark}

\section{Lower bounds}\label{sec:lower}
Next, we turn to lower bounds for commutators and paraproducts. Throughout the section, recall the following for weights $\mu,\lambda,\nu$: 
\begin{itemize} 
\item 
We set $\lambda':=\lambda^{-q'/q}$ and 
$$
[\mu,\lambda']_{B_{p,q'}(\nu)}:= \sup_Q \has{\frac{\mu(Q)}{\nu(Q)}}^{1/p} \has{\frac{\lambda'(Q)}{\nu(Q)}}^{1/q'}.$$
\item
Suppose $\nu^{\frac1p+\frac1{q'}} = \mu^{\frac1p} \lambda^{-\frac1q}$. Then the condition $(\mu,\lambda')\in B_{p,q'}(\nu)$ implies the conditions $\mu \in A_\infty(\nu)$ and $\lambda'\in A_\infty(\nu)$ by Lemma \ref{lemma:ainfinity_wrt_bloom}. 
\end{itemize}

\subsection{Statement and overview of proof for commutators}
We start by proving that,  in the Bloom setting, the sharp maximal function condition
 $$\norm{M^\#_\nu b}_{L^r(\nu)}<+\infty$$
 is necessary for the $L^p(\mu)\to L^q(\lambda)$ boundedness of commutators $[b, T ]$ in the range $1<q<p<\infty$. We will prove this under weaker assumptions on the weights than $\mu \in A_p$, $\lambda \in A_q$.

The main challenge in the Bloom setting for $1<q<p<\infty$ is finding a condition  that is not only sufficient but also necessary. 
The proof of the lower bound builds upon techniques from \cite{HHL2016,hytonen2021,HOS23}, and upon weighted sparse analysis.

\begin{theorem}[Sharp maximal function condition is necessary in the Bloom setting] \label{thm:lowercommutator} Let $1<q<p<\infty$ and set $\frac1r := \frac1q-\frac1p$.  
Let  $b\in L^1_{\loc}$ and let $K$ be a  non-degenerate $\omega$-Calder\'on--Zygmund kernel.  Assume the following:
\begin{itemize}
\item \emph{(Weights)} Let $\mu$ and $\lambda$ be weights and set $\lambda':=\lambda^{-q'/q}$ and $\nu^{\frac1p+\frac1{q'}} := \mu^{\frac1p} \lambda^{-\frac1q}$. Assume that $\mu,\lambda',\nu$ are doubling and $(\mu,\lambda')\in B_{p,q'}(\nu)$. 
\item \emph{(Boundedness)} Assume that $U_b:L^p(\mu)\to L^q(\lambda)$ is a bounded linear operator.
\item \emph{(Off-support bilinear integral form representation)} Assume that 
$$
\int_{\R^d} gU_bf\dx=\int_{\R^d}  \int_{\R^d}  (b(y)-b(x))K(y,x)f(x)g(y) \dx \dy
$$
whenever the functions $f\in L^\infty_c$ and $g\in L^\infty_c$ have supports separated by a positive distance. 
\end{itemize}
Then  
$$
\norm{M^\#_\nu b}_{L^r(\nu)}\lesssim_{\mu,\lambda',\nu} [\mu,\lambda']_{B_{p,q'}(\nu)}[\mu]_{A_\infty(\nu)}^{1/p'}[\lambda']_{A_\infty(\nu)}^{1/q} \norm{U_b}_{L^p(\mu)\to L^q(\nu)},
$$
where the implicit constant depends on the doubling constants of $\mu,\lambda',\nu$.
\end{theorem}
\begin{remark}[Unweighted setting]In the unweighted case $\mu=\lambda=\nu=1$ the theorem recovers \cite[Theorem 2.5.1]{hytonen2021} because of the well-known comparison
$$
\norm{M^\# b}_{L^r}\eqsim \inf_c \norm{b-c}_{L^r}.
$$
\end{remark}

From Theorem \ref{thm:lowercommutator}  we can easily deduce the lower bound in Theorem \ref{thm:maincommutator}.

\begin{corollary}[Lower bound for commutator]\label{corollary:commutator}Let $1<q<p<\infty$, $\mu \in A_p$ and $\lambda \in A_q$. Set $\frac1r := \frac1q-\frac1p$ and $\nu^{1/p+1/q'} := \mu^{1/p} \lambda^{-1/q}$.
Let $T$ be a non-degenerate $\omega$-Calder\'on-Zygmund operator with $\omega$ satisfying the Dini condition, and let $b \in L^1_{\loc}$. Then
$$
     \nrm{M^{\sharp}_{\nu}b}_{L^r(\nu)}\lesssim C_{\mu,\lambda}\nrmb{[b,T]}_{L^p(\mu) \to L^q(\lambda)}.
$$
\end{corollary}
\begin{proof}[Proof of Corollary \ref{corollary:commutator}]This follows from observing that the assumptions of Corollary \ref{corollary:commutator} imply the assumptions of Theorem \ref{thm:lowercommutator}, as follows.
From $\mu \in A_p$ and $\lambda\in A_q$ (or equivalently $\lambda'\in A_{q'}$) it follows that $(\mu,\lambda')\in B_{p,q'}(\nu)$ by Lemma \ref{lemma:muckenhoupt_implies_bloom} and that $\nu \in A_{2r'}$ by Lemma \ref{lemma:muckenhoupt_class_bloom}. In particular, $\mu,\lambda',\nu$ are doubling since they are $A_\infty$-weights. Since $\omega$ satisfies the Dini condition, $T$ is bounded from $L^1$ to $L^{1,\infty}$. Therefore, by Lemma \ref{lemma:commutator_duality}, we have the kernel representation
$$\int f[b,T]g= \int_{\R^d} \int_{\R^d}  (b(x)-b(y))K(x,y)f(y)g(x) \dy \dx$$
whenever the functions $f\in L^\infty_c$ and $g\in L^\infty_c$ have supports separated by a positive distance.
\end{proof}

\begin{proof}[Proof of Theorem \ref{thm:lowercommutator}] 
The proof proceeds in three main steps: (1) discretization of the Lebesgue norm of the sharp maximal function, (2) control of mean oscillations by bilinear forms, and (3) use of a sequential-type testing condition on bilinear forms for general abstract operators. Each of these steps is stated  as a separate lemma. How the lemmas are combined to yield the theorem is detailed in what follows. 

First, we discretize. Fix a parameter $\gamma\in(0,1)$.  By combining Corollary \ref{lemma:maximal_dyadic_nondyadic} and Lemma \ref{lemma:discretization}, we obtain
$$
\norm{M^\#_{\nu}b}_{L^r(\nu)} \eqsim_{\nu} \sup_{\cd}\sup_{\cs\subseteq \cd} \has{\sum_{S\in \cs} \has{ \frac{1}{\nu(S)}\int_S \abs{b-\angles{b}_S}\dx}^r \nu(S) }^{1/r},
$$
where the supremum is taken over all dyadic lattices $\cd$ and over all collections $\mc{S} \subseteq \mc{D}$ that are  $(\gamma,\nu)$-sparse.

Second, we control each oscillation by testing a bilinear form against a pair of test functions. By Lemma \ref{lemma:oscillation_by_operator}, for each cube $S\in\cs$ there exist a cube $\tilde{S}$ with $\ell(\st)\sim \ell(S)$ and $\dist(\st,S)\sim \ell(S)$ and functions $f_S^i,g_{\st}^i$ with $\abs{f_S^i}\lesssim 1_S$, $\abs{g_{\st}^i}\lesssim 1_{\st}$ for $i=1,2$ such that
$$
\int_S \abs{b-\angles{b}_S} \dx \lesssim \sum_{i=1}^2\abss{\int_{\R^d} g_{\st}^i U_b f_S^i\dx }
$$
Therefore,
\begin{equation*}
\begin{split}
&\has{\sum_{S\in \cs} \has{\frac{1}{\nu(S)}\int_S \abs{b-\angles{b}_S}\dx}^r \nu(S) }^{1/r}
\leq \sum_{i=1}^2\has{\sum_{S\in \cs} \abss{\frac{1}{\nu(S)}\int_{\R^d} g_{\st}^i U_b f_S^i \dx}^r \nu(S) }^{1/r}.
\end{split}
\end{equation*}

Finally, the right hand-side, the so-called sequential testing condition on bilinear forms, is bounded by the operator norm by Lemma \ref{lemma:necessary_general}, which completes the proof.
\end{proof}

\subsection{Statement and overview of proof for paraproducts}
The lower bound for dyadic paraproducts is stated and proved as for commutators, except that the structure of dyadic paraproducts, in particular the use of dyadic cubes instead of generic cubes, simplifies estimations.

\begin{theorem}[Necessity in the Bloom setting for paraproducts] \label{theorem:lowerparaproduct4}Let $p,q\in(1,\infty)$, set $\frac1r := \frac1q-\frac1p$ and $\frac\alpha{d} := \frac1p-\frac1q$.  Let $\mu$ and $\lambda$ be weights and set $\lambda':=\lambda^{-q'/q}$ and $\nu^{\frac1p+\frac1{q'}} := \mu^{\frac1p} \lambda^{-\frac1q}$. Let $b \in L^1_{\loc}$, assume that $\Pi_b:L^p(\mu)\to L^q(\lambda)$ is bounded and  $(\mu,\lambda')\in B_{p,q'}(\nu)$. Then the following hold:
\begin{align*}
    \norm{b}_{\BMO_\nu^\alpha}&\lesssim [\mu,\lambda']_{B_{p,q'}(\nu)} \norm{\Pi_b}_{L^p(\mu)\to L^q(\lambda)} && p\leq q,\\
    \norm{M^\#_\nu b}_{L^r(\nu)}&\lesssim [\mu,\lambda']_{B_{p,q'}(\nu)} [\mu]_{A_\infty(\nu)}^{1/p'}[\lambda']_{A_\infty(\nu)}^{1/q} \norm{\Pi_b}_{L^p(\mu)\to L^q(\lambda)} && p>q.
\end{align*}
\end{theorem}

\begin{remark}
In this dyadic context the suprema in the weighted sharp maximal function, the space $\BMO_\nu^\alpha$ and in all the weight characteristics are taken over the dyadic lattice associated with the paraproduct.
\end{remark}
\begin{remark}Note that the assumption $(\mu,\lambda')\in B_{p,q'}(\nu)$ is weaker than the assumption $\mu\in A_p$ and $\lambda\in A_q$ by Lemma \ref{lemma:muckenhoupt_implies_bloom}.
\end{remark}

\begin{proof}[Proof of Theorem \ref{theorem:lowerparaproduct4}] The case $q<p$ follows from combining Lemmas \ref{lemma:discretization},  \ref{lemma:oscillation_paraproduct}, and \ref{lemma:necessary_general}.  Note that no dependence on doubling constants occurs in this dyadic context because of the following: 
\begin{enumerate}[(1)]
    \item The dyadic version of the weighted dyadic sharp maximal function is used.
    \item The pairs of test functions that Lemma \ref{lemma:oscillation_paraproduct} gives to Lemma \ref{lemma:necessary_general} are of the form $(f_{Q_S},g_{Q_S})=(1_S,1_S)$. Therefore, inside the proof of Lemma \ref{lemma:necessary_general}, we can use Lemma \ref{lemma:weighted_basic} (dyadic estimate)  instead of Lemma \ref{lemma:basic_lemma_sparse_generic}  (non-dyadic estimate). 
\end{enumerate}
The case $p\leq q$ follows by combining Lemma \ref{lemma:oscillation_paraproduct} and  Lemma \ref{lemma:testing_lower_triangular}. 
\end{proof}

\subsection{Discretizing the norm of the weighted sharp maximal function}

The first step in our proofs for both commutators and paraproducts is the discretization of the $L^r(\nu)$-norm of the weighted sharp maximal function.

\begin{lemma}[Discretized norm of the weighted sharp maximal function]\label{lemma:discretization}Let $b\in L^1_{\mathrm{loc}}$, $\gamma\in(0,1),$ and $\cd$ a dyadic lattice.  Let $\nu$ be a locally finite Borel measure and $r\in(0,\infty)$. Then
\begin{equation*}
 \nrms{\sup_{Q\in\cd} \frac{1_Q}{\nu(Q)}\int_Q \abs{b-\angles{b}_Q}\dx}_{L^r(\nu)}\\
\eqsim \sup_{\cs\subseteq \cd} \has{ \sum_{S\in \cs} \has{ \frac{1}{\nu(S)}\int_S \abs{b-\angles{b}_S}\dx}^r \nu(S) }^{1/r},
\end{equation*}
where the supremum is taken over all  $(\gamma,\nu)$-sparse collections $\mc{S} \subseteq \mc{D}$.
\end{lemma}\begin{proof}
Since for every cube $Q$
$$
\inf_c \int_Q |b-c|\dx \eqsim \int_Q \abs{b - \ip{b}_Q}\dx,
$$
the claimed conclusion is equivalent to the comparison
\begin{equation*}
 \nrms{\sup_{Q\in\cd} \frac{1_Q}{\nu(Q)}\inf_c \int_Q |b-c|\dx}_{L^r(\nu)}\\
\eqsim \sup_{\cs\subseteq \cd} \has{ \sum_{S\in \cs} \has{ \frac{1}{\nu(S)}\inf_c \int_Q |b-c|\dx}^r \nu(S) }^{1/r}.
\end{equation*}
This comparison follows from a standard stopping time argument. Indeed, the classical stopping time argument of principal cubes is abstracted in \cite[Lemma 2.4]{HHL2016}, whose particular case $\tau_Q:=\inf_c \int_Q |b-c|\dx$ recovers the comparison. 
\end{proof}

\subsection{Controlling oscillations by testing bilinear forms against pairs of test functions}
Mean oscillation can be controlled by testing the bilinear form of the operator against pairs of test functions. This is the only step in the proof of the lower bounds that relies on the concrete structure of the operator. 

For commutators $[b,T]$, the argument for the Beurling transform $T=S$ is classical \cite{coifman1976}. The argument for a very general class of Calder\'on--Zygmund singular kernels, together with a discussion on previous results, can be found in  \cite{hytonen2021} (cf. \cite[Proposition 4.2]{HOS23}):

\begin{lemma}[Oscillations are dominated by testing the commutator against test functions;  {\cite{hytonen2021}}]\label{lemma:oscillation_by_operator} Let $b\in L^1_{\loc}(\br^d,\bc)$ and let $K$ be a non-degenerate $\omega$-Calder\'on--Zygmund kernel. 
Let $\ip{g,[b,T]f}$ denote (as convenient self-explanatory abbreviation) the off-support bilinear form
$$
\ip{g,[b,T]f}:=\int_{\R^d} \int_{\R^d} (b(y)-b(x))K(y,x)f(x)g(y) \dx \dy
$$
for  functions $f\in L^\infty_c$ and $g\in L^\infty_c$ with supports separated by a positive distance. 

 Then, for each cube $Q$, there exist a cube $\tilde{Q}$ with $\ell(\qt)\sim \ell(Q)\sim\dist(\qt,Q)$ and functions $f_Q^i,g_{\qt}^i$ with $\abs{f_Q^i}\leq 1_Q$, $\abs{g_{\qt}^i}\leq 1_{\qt}$ for $i=1,2$ such that
$$
\int_Q \abs{b-\angles{b}_Q} \dx \lesssim  \sum_{i=1}^2\absb{\ip{ g_{\qt}^i, [b,T]f_Q^i } }.
$$
\end{lemma}
In the case of paraproducts, the argument and test functions are particularly simple.
\begin{lemma}[Oscillations are dominated by testing paraproduct against indicators of cubes]\label{lemma:oscillation_paraproduct} Let $b \in L^1_{\loc}$ and let $\mc{D}$ be a dyadic lattice. Then, for every $Q \in \mc{D}$, we have 
$$\int_Q | b-\angles{b}_Q | \dx\lesssim \int_{\R^d} | 1_Q\Pi_{b,\cd}(1_Q)  | \dx.$$
\end{lemma}
\begin{proof} Let $Q\in\cd$. Using the triangle inequality to replace $\ip{b}_Q$ by a term more appropriate for a paraproduct, we obtain formally
\begin{equation*}
    \begin{split}
    \int_Q | b-\angles{b}_Q | \dx
    &\leq 2 \inf_c  \int_Q | b-c | \dx\\
    &\leq 2 \int_Q | b-\angles{b}_Q+\sum_{R\supsetneq Q} D_Rb \angles{1_Q}_R | \dx\\
    &=2 \int_Q | \sum_{R\subseteq Q} D_Rb \angles{1_Q}_R+\sum_{R\supsetneq Q} D_Rb \angles{1_Q}_R| \dx\\
    &= 2 \int_{\R^d} | 1_Q\Pi_b(1_Q)  | \dx.
    \end{split}
\end{equation*}
Rigorously, we consider the truncations 
$$
\cd_{M,N}:=\{Q\in \cd : 2^{-M}\leq \ell(Q)\leq 2^{N} \},
$$
so that the term
$$
\sum_{R\in \cd_{M,N} : R\supsetneq Q} D_Rb \angles{1_Q}_R
$$
is finite. Recall that for each $f\in L^\infty_c$ the sum  $\Pi_{b,\cd}f$ converges unconditionally in $L^1_{\loc}$ by the definition of dyadic paraproducts. Then, by a similar calculation as above, combined with the unconditional convergence in $L^1(Q)$, we obtain
$$
 \int_Q | b-\angles{b}_Q | \dx\leq \lim_{N\to \infty}\lim_{M\to \infty} \int_{\R^d} |1_Q \Pi_{b,\cd_{M,N}} 1_Q|\dx=\int_{\R^d} |1_Q \Pi_{b} 1_Q|\dx.
$$
\end{proof}

\subsection{Sequential testing condition for general operators}
The final steps in our arguments for the lower bounds for paraproducts and commutators is the necessity of a sequential testing condition for general operators. We consider the cases $p<q$ and $p\geq q$ separately.

\begin{lemma}[Sequential testing condition on bilinear form is necessary]\label{lemma:necessary_general}Let $1<q<p<\infty$ and set $\frac{1}{r}:= \frac{1}{p}-\frac{1}{q}$. Let $\mu$ and $\lambda$ be weights and  $U:L^p(\mu)\to L^q(\lambda)$ a bounded linear operator. Assume the following:
\begin{itemize}
\item \emph{(Weights)}  Set $\lambda':=\lambda^{-q'/q}$ and $\nu^{\frac1p+\frac1{q'}} := \mu^{\frac1p} \lambda^{-\frac1q}$. Assume that  $\mu,\lambda',\nu$ are doubling and  $(\mu,\lambda')\in B_{p,q'}(\nu)$. 
\item \emph{(Test functions)} To each cube $P$ assign a cube $Q_P$ with $\ell(Q_P)\lesssim \ell(P)$ and $\dist(Q_P,P)\lesssim \ell(P)$ and a function $f_{Q_P}$ with $\abs{f_{Q_P}}\lesssim 1_{Q_P}$. Similarly, to each cube $P$ assign a cube $R_P$ with $\ell(R_P)\lesssim \ell(P)$ and $\dist(R_P,P)\lesssim \ell(P)$ and a function $g_{R_P}$ with $\abs{g_{R_P}}\lesssim 1_{R_P}$.

\item \emph{(Sparse collection)} Assume that $\cs$ is a $(\delta,\nu)$-sparse collection of cubes.
\end{itemize}
Then
\begin{equation*}
\begin{split}
&\has{\sum_{S\in \cs} \abss{\frac{1}{\nu(S)}\int_{\R^d} g_{R_S} U f_{Q_S}\dx}^r \nu(S) }^{1/r} \\
&\hspace{2cm}\lesssim_{\mu,\lambda',\nu} [\mu,\lambda',\nu]_{B_{p,q'}(\nu)}[\mu]_{A_\infty(\nu)}^{1/p'}[\lambda']_{A_\infty(\nu)}^{1/q} \norm{U}_{L^p(\mu)\to L^q(\lambda)}.
\end{split}
\end{equation*}
The implicit constant depends on the doubling constants of the weights $\mu,\lambda',\nu$.
\end{lemma}

\begin{remark}[Formulation in terms of bilinear form and weak-form boundedness]
As in \cite[Theorem 2.5.1]{hytonen2021}, Lemma \ref{lemma:necessary_general} can be formulated entirely in terms of a bilinear form $I:L^\infty_c\times L^\infty_c\to \bc$. The boundedness then takes the following weaker form: 
$$
\sum_{i=1}^N |I(f_{Q_i},g_{R_i})|\lesssim \nrms{\sum_{i=1}^N \norm{f_{Q_i}}_{L^\infty(Q_i)} 1_{Q_i}}_{L^p(\mu)} \nrms{\sum_{i=1}^N \norm{g_{R_i}}_{L^\infty(R_i)} 1_{R_i}}_{L^{q'}(\lambda')} 
$$
whenever $(f_{Q_i},g_{R_i}) \in L^\infty(Q_i)\times L^\infty(R_i)$ for $i=1,\ldots,N$ and $N\in\bn$. In the case of commutators, the pairs of test functions $(f_{Q_i},g_{Q_i})\in L^\infty(Q_i)\times L^\infty(R_i)$ are such that  $\dist(Q_i,R_i)\sim \ell(Q_i)\sim \ell(R_i)$, as in Lemma \ref{lemma:oscillation_by_operator}.
\end{remark}
\begin{remark}[Weak-form boundedness in case $q\geq p$]When $q\geq p$, the weak-form boundedness becomes
$$
|I(f_{Q},g_{R})|\lesssim \norm{f_{Q}}_{L^\infty(Q)} \norm{g_{R}}_{L^\infty(R)}\mu(Q)^{1/p} \lambda'(Q)^{1/q'}.
$$
For instances of uses of such {\it $L^\infty$-type testing-conditions} in the context of stopping time arguments, see for example \cite{scurry2010,hanninen2016,hanninen2017}. 
\end{remark}

\begin{proof}[Proof of Lemma \ref{lemma:necessary_general}]The argument builds upon ideas from \cite{hytonen2021} and \cite{HOS23}.
By the $\ell^r(\nu)$-$\ell^{r'}(\nu)$ duality, the estimate of the conclusion is equivalent to the estimate
$$
\sum_{S \in \mc{S}} \has{\frac{1}{\nu(S)}\int_{\R^d} g_{R_S} U f_{Q_S}\dx} \gamma_S\nu(S) \lesssim \has{\sum_{S \in \mc{S}} \gamma_S^{r'} \nu(S)}^{1/r'}.
$$
By assumption we have $1=\frac{r'}{q'}+\frac{r'}{p}$, so
$$
\gamma_S\leq [\mu,\lambda']_{B_{p,q'}(\nu)} \has{\gamma_S^{r'} \frac{\nu(S)}{\lambda^{-q'/q}(S)}}^{1/q'} \cdot \left(\gamma_S^{r'}\frac{\nu(S)}{\mu(S)}\right)^{1/p}=:[\mu,\lambda']_{B_{p,q'}(\nu)}\,\beta_S\cdot \alpha_S.
$$
Therefore,
$$
\sum_{S \in \mc{S}} \int_{\R^d} g_{R_S} U f_{Q_S}\dx\, \gamma_S\lesssim_{\mu,\lambda',\nu} \int_{\R^d} \sum_{S \in \mc{S}}  (\beta_S g_{R_S})(U \alpha_S f_{Q_S})\dx.
$$
Randomization is used to decouple the factors in the summands. Let $(\epsilon_S)_{S\in \cs}$ be independent Rademacher variables on a probability space, i.e.  $\be (\epsilon_S\epsilon_T)=\delta_{S,T}$ for all $S,T \in \mc{D}$ and $|\epsilon_S|=1$. Now,
\begin{align*}\int_{\R^d} \sum_{S\in \mc{S}} \beta_Sg_{R_S} U \alpha_S f_{Q_S}\dx&=\be \int_{\R^d} \has{\sum_{T\in \mc{S}}\epsilon_T\beta_Tg_{R_T}} U \has{\sum_{S\in \mc{S}} \epsilon_S \alpha_S f_{Q_S}}\dx\\&=:\be \int \tilde{g}_\epsilon U \tilde{f}_\epsilon \dx.
\end{align*}
By H\"older's inequality and boundedness of $U$ from $L^p(\lambda)$ to $L^q(\mu)$, we have 
\begin{align*}
\int \tilde{g}_\epsilon U \tilde{f}_\epsilon \dx&=\int \hab{\tilde{g}_\epsilon (\lambda')^{-1/q'}} U \hab{\tilde{f}_\epsilon \lambda^{1/q}} \dx\\ &\leq \norm{ \tilde{g}_\epsilon }_{L^{q'}(\lambda')} \norm{U\tilde{f}_\epsilon}_{L^q(\lambda)}\\ &\leq \norm{U}_{L^p(\mu)\to L^q(\lambda)}\norm{ \tilde{g}_\epsilon }_{L^{q'}(\lambda')}  \norm{\tilde{f}_\epsilon}_{L^p(\mu)}. 
\end{align*}
The proof is completed by checking the estimates
\begin{align*}
\norm{\tilde{f}_\epsilon}_{L^p(\mu)}&\leq \has{\sum_S \gamma_S^{r'} \nu(S)}^{1/p},\\
\norm{\tilde{g}_\epsilon}_{L^{q'}(\lambda')}&\leq \has{\sum_S \gamma_S^{r'} \nu(S)}^{1/q'}
\end{align*}
and taking expectations. 

\medskip

We tackle the estimate for $\tilde{f}_\epsilon$; the estimate for $\tilde{g}_\epsilon$ is tackled similarly.  
Let us first write out what suffices to be checked. On the one hand, recalling that $\tilde{f}_\epsilon:=\sum_{S}\epsilon_S \alpha_S f_{Q_S}$ and that by assumption $f_{Q_P}\lesssim 1_{Q_P}$, we see that
$$
\norm{\tilde{f}_\epsilon}_{L^p(\mu)}\lesssim \nrms{\sum_{S\in \mc{S}} \alpha_S 1_{Q_S}}_{L^p(\mu)}.
$$
On the other hand, recalling that
$\alpha_S:=\left(\gamma_S^{r'}\frac{\nu(S)}{\mu(S)}\right)^{1/p}$, we see that
$$
\sum_S \alpha_S^p \mu(S)\leq \sum_S \gamma_S^{r'} \nu(S).
$$
Thus, it suffices to check that
$$
\nrms{\sum_{S\in \mc{S}} \alpha_S 1_{Q_S}}_{L^p(\mu)}\lesssim_{\mu,\nu} \left(\sum_S \alpha_S^p \mu(S)\right)^{1/p}.
$$

This estimate is checked as follows.
By assumption $\dist(Q_S,S)\lesssim \ell(S)$ and $\ell(Q_S)\lesssim \ell(S)$, so there is an $a\geq 1$ such that $a S\supseteq Q_S$. Therefore,
$$
\nrms{\sum_{S\in \mc{S}} \alpha_S 1_{Q_S}}_{L^p(\mu)}\leq\nrms{\sum_{S} \alpha_S 1_{aS}}_{L^p(\mu)}.
$$ 
Since $\cs$ is $(\delta,\nu)$-sparse by assumption, for each $S\in \cs$ there is a $E_S\subseteq S$ such $\nu(E_S)\geq \delta \nu(S)$ and such that the sets $\{E_S\}_{S\in\cs}$ are disjoint. Since $\nu$ is doubling, we have 
$$\nu(E_S)\geq \delta \nu(S)\geq \delta 
 c_{a,\nu}\nu(aS).$$ Since $a\geq 1$, we also have $E_S \subseteq S\subseteq aS$, so $\{aS\}_{S\in\cs}$ is $(\delta  c_{a,\nu},\nu)$-sparse.

Now, since by the preceding $\{aS\}_{S\in\cs}$ is $\nu$-sparse and by assumption $\mu\in A_\infty(\nu)$, it follows by Lemma \ref{lemma:basic_lemma_sparse_generic} that
$$
\nrms{\sum_{S\in \mc{S}} \alpha_S 1_{aS}}_{L^p(\mu)}\lesssim_{\mu,\nu}[\mu]_{A_\infty(\nu)}^{1/p'} \has{\sum_{S\in \mc{S}} \alpha_S^{p} \mu(aS)}^{1/p}. 
$$
Since $\mu$ is doubling, we have $\mu(aS)\lesssim_{a,\mu} \mu(S)$, finishing the proof.
\end{proof}

We end this section with the necessity of testing condition for general operators in the case $p\leq q$, which is much simpler than the case $p>q$.

\begin{lemma}[Testing condition in case {$p\leq q$}, cf. \cite{HOS23}]\label{lemma:testing_lower_triangular}Let $1<p\leq q<\infty$. Let $\mu, \lambda$ and $\nu$ be weights and set $\lambda':=\lambda^{-q'/q}.$ Assume that $U:L^p(\mu)\to L^q(\lambda)$ is bounded and $(\mu,\lambda') \in B_{p,q'}(\nu)$. Then
$$
\sup_Q \frac{1}{\nu(Q)^{1/p+1/q'}}\int_{\R^d} 1_Q\abs{U1_Q} \dx\leq [\mu,\lambda']_{B_{p,q'}(\nu)} \norm{U}_{L^p(\mu)\to L^q(\lambda)}.
$$
\end{lemma}
\begin{proof}By the relation $(\lambda')^{1/q'} \lambda^{1/q}=1$ and  H\"older's inequality,
\begin{equation*}
\begin{split}
\int_{\R^d} 1_Q\abs{U1_Q} \dx&=\int_{\R^d} \hab{(\lambda')^{1/q'}1_Q}\hab{\abs{U1_Q}\lambda^{1/p}} \dx\\
&\leq \norm{1_Q}_{L^{q'}(\lambda')} \norm{U1_Q}_{L^{q}(\lambda)}\\&\leq \norm{U}_{L^p(\mu)\to L^q(\lambda)} \norm{1_Q}_{L^{q'}(\lambda')}\norm{1_Q}_{L^{p}(\mu)}.
\end{split}
\end{equation*}
The claim now follows from the definition of $[\mu,\lambda']_{B_{p,q'}(\nu)}$.
\end{proof}

\section{The multiplier condition}\label{sec:nonneccessary}
Let $1<q<p<\infty$ and set $1/r:=1/q-1/p$. Let $\mu,\lambda$ be weights and set $\nu^{\frac{1}{p}+\frac{1}{q'}}:= \mu^{1/p}\lambda^{-1/q}$. In this final section we will compare the multiplier condition
$$
\inf_c \, \norm{f\mapsto (b-c)f}_{L^p(\mu)\to L^q(\lambda)}=\inf_c\nrm{(b-c)\nu^{-1}}_{L^r(\nu)}<\infty,$$
and sharp maximal condition $M^{\#}_\nu b \in L^r(\nu)$ and  prove that, unlike in the unweighted setting, the multiplier condition is not necessary for the boundedness of the commutator or for the boundedness of the paraproduct when $p>q$.

\subsection{Conditions under which the multiplier and sharp maximal conditions are equivalent}

In this subsection we show that under the assumption that $\nu \in A_{r'}$, the multiplier norm and the $L^r(\nu)$-norm of the sharp maximal function are comparable.
In the next subsection we will see that this is not possible in general. In particular, the comparability result of this subsection shows that in the proof of Theorem \ref{theorem:commutatornormvsmultipliernorm}\ref{it:multinec2} in the next subsection it will be crucial to use a weight $\nu \notin A_{r'}$.

\begin{proposition}\label{prop:feffermansteintyperesult}
Let $b \in L^1_{\loc}$,  $r \in (1,\infty)$ and let $\nu$ be a weight.
\begin{enumerate}[(i)]
\item \label{it:FS1} If $\nu$ is doubling, we have
\[
\norm{M^{\sharp}_{\nu}b}_{L^r(\nu)} \lesssim_\nu \inf_{c}\, \norm{(b-c)\nu^{-1}}_{L^r(\nu)}.
\]
\item \label{it:FS2} If $\nu\in A_{r'}$, we have 
\[
\inf_{c}\norm{(b-c)\nu^{-1}}_{L^r(\nu)}\lesssim [\nu]_{A_{r'}}^{r-1} [\nu]_{A_{\infty}}\norm{M^{\sharp}_{\nu}b}_{L^r(\nu)}.
\]
\end{enumerate}
\end{proposition}

To prove Proposition \ref{prop:feffermansteintyperesult}, we will need the following weighted version of \cite[Lemma 3.6]{hytonen2021}. The proof is almost identical.
\begin{lemma}\label{lemma:weightedanalogue}
Let $b\in L^1_{\loc}$, $1<r<\infty$ and $\nu\in A_{r'}$. If we have for all cubes $Q$ that
\[
\|(b-\ip{b}_Q)\nu^{-1}\|_{L^r(\nu\ind_Q)}\leq C,
\]
then there is a constant $c$ such that
\[
\|(b-c)\nu^{-1}\|_{L^r(\nu)}\leq C.
\]
\end{lemma}
\begin{proof}
Let us consider a sequence of cubes $Q_1\subseteq Q_2\subseteq \cdots$ with $\bigcup_{j=1}^\infty Q_j=\R^d$. For $j\leq k$, we have
\begin{equation*}
\abs{\ip{b}_{Q_k}-\ip{b}_{Q_j}}=\abs{(b(x)-\ip{b}_{Q_j})-(b(x)-\ip{b}_{Q_k})}
\end{equation*}
and hence, multiplying by $\nu^{-1/r'}$, taking the $L^r$-average over $Q_j$ and using Minkowski's inequality, we have
\begin{align*}
\has{\frac{1}{|Q_j|}&\int_{Q_j} \abs{\ip{b}_{Q_k}-\ip{b}_{Q_j}}^r\nu^{1-r}(x)\dx}^{1/r} \\
&\leq |Q_j|^{-1/r}\norm{(b-\ip{b}_{Q_j})\nu^{-1}}_{L^r(\nu\ind_{Q_j})}+|Q_j|^{-1/r}\norm{(b-\ip{b}_{Q_k})\nu^{-1}}_{L^r(\nu \ind_{Q_j})} \\
&\leq C|Q_j|^{-1/r}+|Q_j|^{-1/r}\norm{(b-\ip{b}_{Q_k})\nu^{-1}}_{L^r(\nu \ind_{Q_k})} \\
&\leq 2C|Q_j|^{-1/r}.
\end{align*}
Rearranging the terms in this estimate, we obtain
\begin{equation}\label{eq:weightedanalogue1}
\abs{\ip{b}_{Q_k}-\ip{b}_{Q_j}}\leq 2C\cdot \nu^{1-r}(Q_j)^{-1/r}.
\end{equation}
Note that because $\nu\in A_{r'}$, we know that $\nu^{1-r}\in A_r$. Furthermore, any $A_\infty$ weight gives infinite measure to $\R^d$ and thus $\nu^{1-r}(Q_j)^{1-r}\to 0$ as $j\to \infty$. Thus inequality \eqref{eq:weightedanalogue1} implies that $(\ip{b}_{Q_j})_{j=1}^\infty$ is a Cauchy sequence and thus converges to some $c$. We conclude by Fatou's lemma that 
\begin{align*}
\int_{\R^d} |b-c|^r\nu^{1-r} \dd x&=\int_{\R^d} \lim_{j\to\infty}1_{Q_j}|b-\ip{b}_{Q_j}|^r\nu^{1-r} \dd x\\&\leq \liminf_{j\to\infty}\int_{Q_j}|b-\ip{b}_{Q_j}|^r\nu^{1-r}\dd x\leq C^r.
\end{align*}
This finishes the proof.
\end{proof}

We are now ready to prove Proposition \ref{prop:feffermansteintyperesult}.
\begin{proof}[Proof of Proposition \ref{prop:feffermansteintyperesult}]
\ref{it:FS1}: Let $c$ be a constant. Note that for $x \in \R^d$ we have 
\begin{align*}
M^\#_\nu b(x) \leq 2 \sup_{Q \ni x} \frac{1}{\nu(Q)}\int_Q (b-c)\dd x =2  \,M^{\nu} \hab{(b-c)\nu^{-1}}(x).
\end{align*}
So the claim follows from the Hardy--Littlewood maximal inequality. The non-dyadic versions of the maximal functions are used in this context. Therefore, due to the one-third trick, the constant in the Hardy--Littlewood maximal inequality depends on the doubling constant of the measure $\nu$.

For \ref{it:FS2} it 
suffices to estimate
\[
\norm{(b-\ip{b}_Q)\nu^{-1}}_{L^r(\nu \ind_Q)}=\norm{(b-\ip{b}_Q)1_Q}_{L^r(\nu^{1-r})}
\]
uniformly over all cubes $Q$ by Lemma \ref{lemma:weightedanalogue}. Fix a cube $Q$. By, e.g., \cite[Lemma 3.1.2]{hytonen2021}, there exists a $\frac12$-sparse family $\mc{S}$ such that
$$
\ind_Q \abs{b-\ip{b}_Q} \lesssim \sum_{S \in \mc{S}}\frac{1}{\abs{S}}\int_R \abs{b-\ip{b}_S}\dd x.
$$
Let $g\in L^{r'}(\nu)$ be positive. By H\"older's inequality, we estimate
\begin{align*}
\int_Q |b(x)-&\ip{b}_{Q}|g(x) \dx \\
&\lesssim \sum_{S\in \mc{S}} \frac{1}{\nu(S)} \has{\int_S\abs{b-\ip{b}_S}\dx} \ip{g}_S \nu(S) \\
&\leq \has{\sum_{S\in \mc{S}}  \has{ \frac{1}{\nu(S)}\int_S \abs{b-\ip{b}_S}\dx }^{r} \nu(S) }^{1/r} \has{\sum_{S\in \mc{S}}  \ip{g}_S^{r'} \nu(S)}^{1/r'}.
\end{align*}
For the first term, we estimate using Corollary \ref{cor:Ainftysparse} 
\begin{align*}
\has{\sum_{S\in \mc{S}} \has{\frac{1}{\nu(S)}\int_S \abs{b-\ip{b}_S}\dx }^{r} \nu(S) }^{1/r}&\lesssim [\nu]_{A_{\infty}}^{1/r}\has{\sum_{S\in \mc{S}} \has{ \frac{1}{\nu(S)}\int_S \abs{b-\ip{b}_S}\dx }^{r} \nu(E_S) }^{1/r} \\
&=[\nu]_{A_{\infty}}^{1/r}\has{\sum_{S\in \mc{S}} \int_{E_S}\has{ \frac{1}{\nu(S)}\int_{S} \abs{b-\ip{b}_S}\dx }^{r} \dnu }^{1/r} \\
&\leq [\nu]_{A_{\infty}}^{1/r}\norm{M^{\sharp}_{\nu}b}_{L^r(\nu)}.
\end{align*}
For the second term, we have using Corollary \ref{cor:Ainftysparse} and Lemma \ref{lemma:strongtypemaximalnondyadic} that
\begin{align*}
\has{\sum_{S \in \mc{S}} \ip{g}_S^{r'} \nu(S)}^{1/r'}\lesssim [\nu]_{A_{\infty}}^{1/r'}\has{\sum_{S \in \mc{S}} \ip{g}_S^{r'} \nu(E_S)}^{1/r'}
&=[\nu]_{A_{\infty}}^{1/r'}\has{\sum_S \int_{E_S}\ip{g}_S^{r'}\dnu}^{1/r'} \\
&\leq [\nu]_{A_{\infty}}^{1/r'} \norm{Mg}_{L^{r'}(\nu)} \\
&\lesssim [\nu]_{A_{\infty}}^{1/r'} [\nu]_{A_{r'}}^{\frac{1}{r'-1}}\|g\|_{L^{r'}(\nu)}.
\end{align*}
Combining these estimates and using duality, we get
\begin{align*}
\|(b-\ip{b}_Q)1_Q\|_{L^r(\nu^{1-r})}\lesssim [\nu]_{A_{r'}}^{r-1} [\nu]_{A_{\infty}}\norm{M^{\sharp}_{\nu}b}_{L^r(\nu)},
\end{align*}finishing the proof.
\end{proof}

\subsection{Multiplier condition is in general non-necessary for boundedness} 
In the previous subsection, we have seen that the pointwise multiplication condition and the sharp maximal condition are equivalent for $\nu \in A_{r'}$, which combined with Theorem \ref{thm:maincommutator} yields the following. We have
\begin{equation}\label{eq:counterpositive}
    \norm{[b,T]}_{L^p(\mu)\to L^q(\lambda)}\eqsim_{\mu,\lambda} \inf_c \, \norm{(b-c)\nu^{-1}}_{L^r(\nu)},\quad\text{assuming $\nu \in A_{r'}$,} 
\end{equation}
for $1<q<p<\infty$, $\frac1r:=\frac1q-\frac1p$, $\mu \in A_p$, $\lambda \in A_q$ and $\nu^{\frac{1}{p}+\frac{1}{q'}}:={\mu^{1/p}}{\lambda^{-1/q}}$, which is a direct analog of the unweighted setting in \cite{hytonen2021}.
In this subsection we will prove that the multiplier condition
\begin{equation*}
\inf_c\nrm{(b-c)\nu^{-1}}_{L^r(\nu)}<\infty
\end{equation*}
is in general not necessary for the $L^p(\mu)\to L^q(\lambda)$-boundedness of commutators or paraproducts.  From \eqref{eq:counterpositive}, we know that a counterexample must satisfy $\nu \in A_{2r'}\setminus A_{r'}$.

\begin{theorem}\label{theorem:commutatornormvsmultipliernorm}
Let $T$ be an $\omega$-Calder\'on-Zygmund operator with $\omega$ satisfying the Dini condition.
\begin{enumerate}[(i)]
\item\label{it:multinec1} \emph{(Sufficiency)} Suppose $1<p,q<\infty$, $\mu\in A_p$ and $\lambda\in A_q$. For every $b\in L^1_{\loc}$,
\begin{align*}
\norm{[b,T]}_{L^p(\mu)\to L^q(\lambda)}&\leq \big(\norm{T}_{L^p(\mu)\to L^p(\mu)}+\norm{T}_{L^q(\lambda)\to L^q(\lambda)}\big) \\&\hspace{1cm} \cdot \inf_c \, \norm{f\mapsto (b-c)f}_{L^p(\mu)\to L^q(\lambda)}.
\end{align*}

\item \label{it:multinec2}\emph{(Non-necessity)} Suppose $1<q<p<\infty$. Then there are $\mu\in A_p$, $\lambda\in A_q$ and $b\in L^1_{\loc}$ such that
\begin{align*}
\inf_c \, \norm{f\mapsto (b-c)f}_{L^p(\mu)\to L^q(\lambda)}&=\infty, \\\norm{[b,T]}_{L^p(\mu)\to L^q(\lambda)}&<\infty.
\end{align*}
\end{enumerate}
\end{theorem}
\begin{proof}
 The proof of \ref{it:multinec1}  is straightforward and well-known. Let $c$ be a constant and denote $\Theta(c):=\norm{f\mapsto (b-c)f}_{L^p(\mu)\to L^q(\lambda)}$.  Then
\begin{align*}
    \norm{[b,T]f}_{L^q(\lambda)} &=\norm{[b-c,T]f}_{L^q(\lambda)} \\
    &\leq \norm{(b-c)Tf}_{L^q(\lambda)}+\norm{T((b-c)f)}_{L^q(\lambda)} \\
    &\leq \Theta(c)\norm{Tf}_{L^p(\mu)}+\norm{T}_{L^q(\lambda)\to L^q(\lambda)} \norm{(b-c)f}_{L^q(\lambda)} \\
    &\leq \Theta(c)\big(\norm{Tf}_{L^p(\mu)}+\norm{T}_{L^q(\lambda)\to L^q(\lambda)}\norm{f}_{L^p(\mu)}\big) \\
    &\leq \Theta(c)\big(\norm{T}_{L^p(\mu)\to L^p(\mu)}+\norm{T}_{L^q(\lambda)\to L^q(\lambda)}\big)\norm{f}_{L^p(\mu)}.
\end{align*}

For \ref{it:multinec2} the idea is to use Lemma \ref{lemma:keylemmaforcommuvsmult} below. Let $b\in L^\infty_c$ be given by $b := \ind_{B(0,1)}$, where $B(0,1)$ is the unit ball. Set $\frac1r:=\frac1q-\frac1p$,  let $\gamma:=d(r'-1)$ and define $\nu(x):=|x|^\gamma$. Then $\nu\in A_{2r'}$ since $-d<\gamma< d(2r'-1)$. By Lemma \ref{lemma:muckenhoupt_class_bloom}, there exist power weights $\mu\in A_p$ and $\lambda\in A_q$ such that $\nu^{1/r'}=\mu^{1/p}\lambda^{-1/q}$. By Theorem \ref{theorem:uppercommutator} and Lemma \ref{lemma:keylemmaforcommuvsmult}\ref{it:keynec1} we deduce that 
\begin{equation*}
\|[b,T]\|_{L^p(\mu)\to L^q(\lambda)}\lesssim_{\mu,\lambda}\norm{M_\nu^{\#}b}_{L^r(\nu)}<\infty.
\end{equation*}
Also, because for all constants $c$ we have by H\"{o}lder's inequality that
\begin{equation}\label{eq:holderappliedtomultiplier}
\norm{f\mapsto (b-c)f}_{L^p(\mu)\to L^q(\lambda)}= \norm{(b-c)\nu^{-1}}_{L^r(\nu)},
\end{equation}
Lemma \ref{lemma:keylemmaforcommuvsmult}\ref{it:keynec2} shows that 
\begin{equation*}
\inf_c \norm{f\mapsto (b-c)f}_{L^p(\mu)\to L^q(\lambda)}=\infty.\qedhere
\end{equation*}
\end{proof}

We get a similar result to Theorem \ref{theorem:commutatornormvsmultipliernorm} for paraproducts instead of commutators. The sufficiency is an immediate consequence of Theorem \ref{theorem:bloompara}, Proposition \ref{prop:feffermansteintyperesult}\ref{it:FS1} and equality \eqref{eq:holderappliedtomultiplier}. The non-necessity part follows by using Theorem \ref{theorem:bloompara} instead of Theorem \ref{theorem:uppercommutator}, in an otherwise identical proof to the non-necessity part of Theorem \ref{theorem:commutatornormvsmultipliernorm}. We write this result as follows:

\begin{theorem}\label{theorem:paraproductnormvsmultipliernorm}
Let $1<q<p<\infty$.
\begin{enumerate}[(i)]
\item\emph{(Sufficiency)} Suppose $\mu\in A_p$ and $\lambda\in A_q$. For every $b\in L^1_{\loc}$,
\begin{align*}
\norm{\Pi_b}_{L^p(\mu) \to L^q(\lambda)} &\lesssim [\mu]_{A_p}^{\frac1{p-1}} [\lambda]_{A_q} [\nu]_{A_\infty}^{1/q-1/p} \cdot \inf_c \, \norm{f\mapsto (b-c)f}_{L^p(\mu)\to L^q(\lambda)}.
\end{align*}
\item \emph{(Non-necessity)}  There are $\mu\in A_p$, $\lambda\in A_q$ and $b\in L^1_{\loc}$ such that
\begin{align*}
\inf_c \, \norm{f\mapsto (b-c)f}_{L^p(\mu)\to L^q(\lambda)}&=\infty, \\ \norm{\Pi_b}_{L^p(\mu) \to L^q(\lambda)}&<\infty.
\end{align*}
\end{enumerate}
\end{theorem}

We will prove the non-necessity using power weights. We start by documenting a useful property of power weights with a non-negative power:
\begin{lemma}\label{lemma:lowerboundforweightofcube}
Suppose $\gamma\geq 0$ and $\nu(x):=|x|^\gamma$. Then  we have for all cubes $Q$
\[
\ell(Q)^{\gamma+d}\lesssim \nu(Q)
\]
In particular, for all $\alpha>0$ we have
\[
\sup_{Q : \ell(Q)\geq \alpha}\frac{|Q|}{\nu(Q)}\lesssim \alpha^{-\gamma}.
\]
\end{lemma}
\begin{proof}
Let $Q$ be a cube. Note that, since $\gamma \geq 0$, we have $\nu(Q) \geq \nu(Q')$, where $Q'$ is the cube with center $0$ and side length $\ell(Q)$. Integrating in polar coordinates shows that 
$
\nu(Q') \gtrsim \ell(Q)^{\gamma+d}
$
proving the first estimate.
The second estimate is an immediate corollary of the first.
\end{proof}

The following lemma is the key in the proof of the non-necessity. 

\begin{lemma}\label{lemma:keylemmaforcommuvsmult}
Let $r \in (1,\infty)$ and set $b := \ind_{B(0,1)}$, where $B(0,1)$ is the unit ball.
\begin{enumerate}[(i)]
\item \label{it:keynec1} Let $\nu(x):=|x|^\gamma$ for $\gamma\geq 0$. Then 
\begin{equation*}
\norm{M^{\sharp}_{\nu}b}_{L^r(\nu)}<\infty.
\end{equation*}
\item \label{it:keynec2} Let $\nu(x):=|x|^{d(r'-1)}$. Then
\begin{equation*}
\inf_c \, \norm{(b-c)\nu^{-1}}_{L^r(\nu)}=\infty.
\end{equation*}
\end{enumerate}
\end{lemma}

\begin{proof}
We use the notation $B(a,s)=\{x\in \R^d : |x-a|\leq s\}$. For \ref{it:keynec1} we write 
\[
\int_{\R^d} (M^{\sharp}_{\nu}b)^r\nu \dx= \Big(\int_{B(0,\frac12)}+\int_{B(0,2)\setminus B(0,\frac12)}+\int_{\R^d\setminus B(0,2)}\Big)(M^{\sharp}_{\nu}b)^r\nu \dx,
\]
and estimate $M^{\sharp}_{\nu}b$ separately in each of the three domains.

For $|x|\leq 1/2$ suppose $Q\ni x$. If $\ell(Q)<\frac{1}{2\sqrt{d}}$, then $Q\subseteq B(0,1)$ and thus 
\[
\int_Q|b-\ip{b}_Q|\dx=0.
\]
For cubes $Q$ with $\ell(Q)\geq \frac{1}{2\sqrt{d}}$, we use the estimate 
\[
\int_Q|b-\ip{b}_Q|\dx\leq 2|Q|,
\]
which yields
\[
M_\nu^{\#}b(x)\leq 2\sup_{Q\ni x: \ell(Q)\geq \frac{1}{2\sqrt{d}}}\frac{|Q|}{\nu(Q)}.
\]
An application of Lemma \ref{lemma:lowerboundforweightofcube} shows that 
$
M_\nu^{\#}b(x)\lesssim 1.
$

For $1/2<|x|\leq 2$ we again estimate
\[
M_\nu^{\#}b(x)\leq 2\sup_{Q\ni x}\frac{|Q|}{\nu(Q)}.
\]
We split into two cases: $\ell(Q) \leq \frac{1}{4\sqrt{d}}$ and $\ell(Q)\geq \frac{1}{4\sqrt{d}}$. In the first case, we have $Q\subseteq \{y \in \R^d\colon |y|\geq 1/4\}$ and therefore 
\[
|Q|\leq 4^\gamma \nu(Q).
\]
In the second case, we use Lemma \ref{lemma:lowerboundforweightofcube} to get
\[
{|Q|}\lesssim \ell(Q)^{\gamma+d}\lesssim {\nu(Q)} .
\]
Combining these estimates, we again obtain $M_\nu^{\#}b(x)\lesssim 1.$

For $|x|> 2$ note first that for any cube $Q$,
\[
\int_Q|b-\ip{b}_Q|\dx\leq 2\int_Q|b|\dx\leq 2\int_{\R^d}|b|\dx\lesssim 1.
\]
If $Q\cap B(0,1)=\emptyset$, we trivially have
\[
\int_Q|b-\ip{b}_Q|\dx=0.
\]
Thus, by Lemma \ref{lemma:lowerboundforweightofcube}, we get 
\[
M_\nu^{\#}b(x)\lesssim \sup_{Q\ni x,Q\cap B(0,1)\neq\emptyset}\frac{1}{\nu(Q)}\lesssim \sup_{Q\ni x,Q\cap B(0,1)\neq\emptyset}\frac{1}{\ell(Q)^{\gamma+d}}.
\]
Note that for all cubes $Q$ in the supremum we have
\[
\ell(Q)\sqrt{d}\geq |x|-1\geq \tfrac{1}{2}|x|
\]
and thus $\ell(Q)\gtrsim |x|$. This yields 
\[
M_\nu^{\#}b(x)\lesssim \frac{1}{|x|^{\gamma+d}}.
\]

The three domains have been considered and using the achieved pointwise estimates respectively, we get
\[
\int_{\R^d} (M^{\sharp}_{\nu}b)^r\nu \dx \lesssim \int_{B(0,2)}|x|^\gamma\dx+\int_{\R^d\setminus B(0,2)}|x|^{-r(\gamma+d)+\gamma}\dx.
\]
Because $\gamma\geq 0$ and $r>1$, we have $-r(\gamma+d)+\gamma = -\gamma(r-1) -dr< -d$. Therefore both the first and the second integral is finite.

\medskip
 
For \ref{it:keynec2} suppose first that $c=0$. Then 
\[
\|(b-c)\nu^{-1}\|_{L^r(\nu)}^r\geq \int_{B(0,1)}|x|^{d(r'-1)(1-r)}\dx=\infty,
\]
because $d(r'-1)(1-r)=-d$. Next, suppose  that $c\neq 0$. Then 
\[
\|(b-c)\nu^{-1}\|_{L^r(\nu)}^r\geq \abs{c}^r\int_{\R^d\setminus B(0,1)}|x|^{d(r'-1)(1-r)}\dx=\infty,
\]
again because $d(r'-1)(1-r)=-d$. 
\end{proof}

\section*{Acknowledgements}
{The authors thank Tuomas Hyt\"onen and Tuomas Oikari for helpful and inspiring discussions. Furthermore, the authors express their gratitude to Pascal Auscher and Pierre Portal for help with the abstract in French. All the authors are supported by the Academy of Finland (through Projects 332740 and 336323).
}
\bibliographystyle{alpha}
\bibliography{manuscript_bib}

\newcommand{\etalchar}[1]{$^{#1}$}
\begin{thebibliography}{HNVW16}

\bibitem[ALM23]{ALM22}
E.~Airta, K.~Li, and H.~Martikainen.
\newblock Two-weight inequalities for multilinear commutators in product spaces.
\newblock {\em Potential Anal.}, 59(4):1745--1792, 2023.

\bibitem[Bar19]{barron2019}
A.~Barron.
\newblock {\em Sparse bounds in harmonic analysis and semiperiodic estimates}.
\newblock PhD thesis, Brown University, 2019.

\bibitem[Blo85]{Bloom1985}
S.~Bloom.
\newblock A commutator theorem and weighted {BMO}.
\newblock {\em Trans. Amer. Math. Soc.}, 292(1):103--122, 1985.

\bibitem[Buc93]{Bu93}
S.~Buckley.
\newblock Estimates for operator norms on weighted spaces and reverse {J}ensen inequalities.
\newblock {\em Trans. Amer. Math. Soc.}, 340(1):253--272, 1993.

\bibitem[Bur79]{Bu79}
D.L. Burkholder.
\newblock A sharp inequality for martingale transforms.
\newblock {\em Ann. Probab.}, 7(5):858--863, 1979.

\bibitem[Chu11]{Ch11}
D.~Chung.
\newblock Sharp estimates for the commutators of the {H}ilbert, {R}iesz transforms and the {B}eurling-{A}hlfors operator on weighted {L}ebesgue spaces.
\newblock {\em Indiana Univ. Math. J.}, 60(5):1543--1588, 2011.

\bibitem[CO17]{cascante2017}
C.~Cascante and J.M. Ortega.
\newblock Two-weight norm inequalities for vector-valued operators.
\newblock {\em Mathematika}, 63(1):72--91, 2017.

\bibitem[CPP12]{CPP12}
D.~Chung, M.C. Pereyra, and C.~Perez.
\newblock Sharp bounds for general commutators on weighted {L}ebesgue spaces.
\newblock {\em Trans. Amer. Math. Soc.}, 364(3):1163--1177, 2012.

\bibitem[CRW76]{coifman1976}
R.~R. Coifman, R.~Rochberg, and G.~Weiss.
\newblock Factorization theorems for {H}ardy spaces in several variables.
\newblock {\em Ann. of Math. (2)}, 103(3):611--635, 1976.

\bibitem[FH18]{FH18}
S.~Fackler and T.P. Hyt\"{o}nen.
\newblock Off-diagonal sharp two-weight estimates for sparse operators.
\newblock {\em New York J. Math.}, 24:21--42, 2018.

\bibitem[FS72]{FS72}
C.~Fefferman and E.M. Stein.
\newblock {$H^{p}$} spaces of several variables.
\newblock {\em Acta Math.}, 129(3-4):137--193, 1972.

\bibitem[H{\"{a}}n17]{hanninen2017}
T.S. H{\"{a}}nninen.
\newblock Two-weight inequality for operator-valued positive dyadic operators by parallel stopping cubes.
\newblock {\em Israel J. Math.}, 219(1):71--114, 2017.

\bibitem[H{\"{a}}n18]{hanninen2018}
T.S. H{\"{a}}nninen.
\newblock Equivalence of sparse and {C}arleson coefficients for general sets.
\newblock {\em Ark. Mat.}, 56(2):333--339, 2018.

\bibitem[HF23]{FH22}
I.~{Holmes Fay} and V.~Fragkiadaki.
\newblock Paraproducts, {B}loom {BMO} and sparse {BMO} functions.
\newblock {\em Rev. Mat. Iberoam.}, 39(6):2079--2118, 2023.

\bibitem[HH16]{hanninen2016}
T.S. H\"{a}nninen and T.P. Hyt\"{o}nen.
\newblock Operator-valued dyadic shifts and the {$T(1)$} theorem.
\newblock {\em Monatsh. Math.}, 180(2):213--253, 2016.

\bibitem[HHL16]{HHL2016}
T.S. H\"{a}nninen, T.P. Hyt\"{o}nen, and K.~Li.
\newblock Two-weight {$L^p$}-{$L^q$} bounds for positive dyadic operators: unified approach to {$p\leq q$} and {$p>q$}.
\newblock {\em Potential Anal.}, 45(3):579--608, 2016.

\bibitem[HL25]{HL25}
E.A. {Honig} and E.~{Lorist}.
\newblock Optimization algorithms for {C}arleson and sparse collections of sets.
\newblock arXiv:2501.07943, 2025.

\bibitem[HLTY23]{HLTY22}
T.~Hyt\"onen, K.~Li, J.~Tao, and D.~Yang.
\newblock The {$L^p$}-to-{$L^q$} compactness of commutators with {$p>q$}.
\newblock {\em Studia Math.}, 271(1):85--105, 2023.

\bibitem[HLW16]{HLW16}
I.~Holmes, M.T. Lacey, and B.D. Wick.
\newblock Bloom's inequality: commutators in a two-weight setting.
\newblock {\em Arch. Math. (Basel)}, 106(1):53--63, 2016.

\bibitem[HLW17]{HLW17}
I.~Holmes, M.T. Lacey, and B.D. Wick.
\newblock Commutators in the two-weight setting.
\newblock {\em Math. Ann.}, 367(1-2):51--80, 2017.

\bibitem[HNVW16]{HNVW16}
T.P. Hyt\"{o}nen, J.M.A.M.~van Neerven, M.C. Veraar, and L.~Weis.
\newblock {\em Analysis in {B}anach spaces. {V}ol. {I}. {M}artingales and {L}ittlewood-{P}aley theory}, volume~63 of {\em Ergebnisse der Mathematik und ihrer Grenzgebiete.}
\newblock Springer, Cham, 2016.

\bibitem[HOS23]{HOS23}
T.~Hyt\"onen, T.~Oikari, and J.~Sinko.
\newblock Fractional {B}loom boundedness and compactness of commutators.
\newblock {\em Forum Mathematicum}, 35(3):809--830, 2023.

\bibitem[HP13]{HP13}
T.P. Hyt\"onen and C.~P\'erez.
\newblock Sharp weighted bounds involving {$A_\infty$}.
\newblock {\em Anal. PDE}, 6(4):777--818, 2013.

\bibitem[HV21]{hanninen2021}
T.S. H\"{a}nninen and I.E. Verbitsky.
\newblock On two-weight norm inequalities for positive dyadic operators.
\newblock {\em Potential Anal.}, 55(2):229--249, 2021.

\bibitem[Hyt12]{Hy12}
T.P. Hyt\"onen.
\newblock The sharp weighted bound for general {C}alder\'on-{Z}ygmund operators.
\newblock {\em Ann. of Math.}, 175(3):1473--1506, 2012.

\bibitem[Hyt17]{Hy17}
T.P. Hyt\"{o}nen.
\newblock Representation of singular integrals by dyadic operators, and the {$A_2$} theorem.
\newblock {\em Expo. Math.}, 35(2):166--205, 2017.

\bibitem[Hyt21a]{hytonen2021}
T.P. Hyt\"{o}nen.
\newblock The {$L^p$}-to-{$L^q$} boundedness of commutators with applications to the {J}acobian operator.
\newblock {\em J. Math. Pures Appl. (9)}, 156:351--391, 2021.

\bibitem[Hyt21b]{Hy21c}
T.P. Hyt\"{o}nen.
\newblock Of commutators and {J}acobians.
\newblock In {\em Geometric aspects of harmonic analysis}, volume~45 of {\em Springer INdAM Ser.}, pages 455--466. Springer, Cham, 2021.

\bibitem[Jan78]{Janson1978}
S.~Janson.
\newblock Mean oscillation and commutators of singular integral operators.
\newblock {\em Ark. Mat.}, 16:263--270, 1978.

\bibitem[Lin17]{Li17c}
S.~Lindberg.
\newblock On the {H}ardy space theory of compensated compactness quantities.
\newblock {\em Arch. Ration. Mech. Anal.}, 224(2):709--742, 2017.

\bibitem[Lin23]{Li22}
S.~Lindberg.
\newblock A note on the {J}acobian problem of {C}oifman, {L}ions, {M}eyer and {S}emmes.
\newblock {\em J. Fourier Anal. Appl.}, 29(6):Paper No. 68, 29, 2023.

\bibitem[LN18]{LN15}
A.K. Lerner and F.~Nazarov.
\newblock Intuitive dyadic calculus: The basics.
\newblock {\em Expositiones Mathematicae}, 2018.

\bibitem[LOP{\etalchar{+}}09]{LOPTT09}
A.K. Lerner, S.~Ombrosi, C.~P\'erez, R.H. Torres, and R.~Trujillo-Gonz\'alez.
\newblock New maximal functions and multiple weights for the multilinear {C}alder\'on-{Z}ygmund theory.
\newblock {\em Adv. Math.}, 220(4):1222--1264, 2009.

\bibitem[LOR17]{LOR17}
A.K. Lerner, S.~Ombrosi, and I.P. {Rivera-R\'{\i}os}.
\newblock On pointwise and weighted estimates for commutators of {C}alder\'{o}n-{Z}ygmund operators.
\newblock {\em Adv. Math.}, 319:153--181, 2017.

\bibitem[LOR21]{LOR21}
A.K. Lerner, S.~Ombrosi, and I.P. {Rivera-R\'{\i}os}.
\newblock On two weight estimates for iterated commutators.
\newblock {\em J. Funct. Anal.}, 281(8):Paper No. 109153, 46, 2021.

\bibitem[Neh57]{Nehari1957}
Z.~Nehari.
\newblock On bounded bilinear forms.
\newblock {\em Ann. of Math.}, 65:153--162, 1957.

\bibitem[Oki92]{okikiolu1992}
K.~Okikiolu.
\newblock Characterization of subsets of rectifiable curves in {${\bf R}^n$}.
\newblock {\em J. London Math. Soc. (2)}, 46(2):336--348, 1992.

\bibitem[Per19]{Pe19}
M.C. Pereyra.
\newblock Dyadic harmonic analysis and weighted inequalities: the sparse revolution.
\newblock In {\em New trends in applied harmonic analysis. {V}ol. 2---harmonic analysis, geometric measure theory, and applications}, Appl. Numer. Harmon. Anal., pages 159--239. Birkh\"{a}user/Springer, Cham, 2019.

\bibitem[PP20]{pau2020}
J.~Pau and A.~Per\"{a}l\"{a}.
\newblock A {T}oeplitz-type operator on {H}ardy spaces in the unit ball.
\newblock {\em Trans. Amer. Math. Soc.}, 373(5):3031--3062, 2020.

\bibitem[Rey24]{rey2022}
G.~Rey.
\newblock Greedy approximation algorithms for sparse collections.
\newblock {\em Publ. Mat.}, 68(1):251--265, 2024.

\bibitem[Scu10]{scurry2010}
J.~Scurry.
\newblock A characterization of two-weight inequalities for a vector-valued operator.
\newblock arXiv:1007.3089, 2010.

\bibitem[ST93]{ST93b}
C.~Segovia and J.L. Torrea.
\newblock Higher order commutators for vector-valued {C}alder\'{o}n-{Z}ygmund operators.
\newblock {\em Trans. Amer. Math. Soc.}, 336(2):537--556, 1993.

\bibitem[Tor86]{To86}
A.~Torchinsky.
\newblock {\em Real-variable methods in harmonic analysis}, volume 123 of {\em Pure and Applied Mathematics}.
\newblock Academic Press, Inc., Orlando, FL, 1986.

\bibitem[Ver96]{verbitsky1996}
I.E. Verbitsky.
\newblock Imbedding and multiplier theorems for discrete {L}ittlewood-{P}aley spaces.
\newblock {\em Pacific J. Math.}, 176(2):529--556, 1996.

\end{thebibliography}

\end{document}